\documentclass[10pt,leqno]{article}
\pdfoutput=1
\usepackage{graphicx}
\usepackage[english]{babel}
\usepackage{tikz}
\usepackage{amscd,epsfig}

\usepackage[style=alphabetic]{biblatex}

\usepackage{footnote}
\usepackage{stmaryrd}
\usepackage{epigraph}
\usepackage{geometry}
\usepackage{bbold}
\usepackage{tikz-cd}

\usepackage{authblk}
\usepackage{amssymb,amsthm,amsmath}
\usepackage{paralist}
\usepackage{indentfirst,csquotes}

\title{Pre-Lie algebras with divided powers and the Deligne groupoid in positive characteristic}
\date{November 1, 2023}
\author{Marvin VERSTRAETE}

\AtEndDocument{%
  \par
  \medskip
  \begin{tabular}{@{}l@{}}%
    \textsc{Univ. Lille, CNRS, UMR 8524 - Laboratoire Paul Painlevé, F-59000 Lille, France}\\
    \textit{E-mail address}: \texttt{marvin.verstraete@univ-lille.fr}
  \end{tabular}}

\begin{document}

\newtheorem{defi}{Definition}[section]
\newtheorem{thm}[defi]{Theorem}
\newtheorem{thmA}{Theorem}
\newtheorem*{thm*}{Theorem}
\newtheorem{cor}[defi]{Corollary}
\newtheorem{remarque}[defi]{Remark}
\newtheorem{lm}[defi]{Lemma}
\newtheorem{prop}[defi]{Proposition}
\newtheorem{algo}[defi]{Algorithme}
\newtheorem{cons}[defi]{Construction}
\newtheorem{propdef}[defi]{Proposition-Definition}

\setcounter{tocdepth}{2}

\renewcommand*{\thethmA}{\Alph{thmA}}

\maketitle

\newcommand{\antishriek}{\mbox{\footnotesize{\rotatebox[origin=c]{180}{$!$}}}}

\let\thefootnote\relax

\begin{abstract}
The purpose of this paper is to develop a deformation theory controlled by pre-Lie algebras with divided powers over a ring of positive characteristic. We show that every differential graded pre-Lie algebra with divided powers comes with operations, called weighted braces, which we use to generalize the classical deformation theory controlled by Lie algebras over a field of characteristic $0$. Explicitly, we define the Maurer-Cartan set, as well as the gauge group, and prove that there is an action of the gauge group on the Maurer-Cartan set. This new deformation theory moreover admits a Goldman-Millson theorem which remains valid over the integers. As an application, we give the computation of the $\pi_0$ of a mapping space $\text{\normalfont Map}(B^c(\mathcal{C}),\mathcal{P})$ with $\mathcal{C}$ and $\mathcal{P}$ suitable cooperad and operad respectively.
\end{abstract}

\tableofcontents

\section*{Introduction}
\addcontentsline{toc}{section}{Introduction}

An important result in deformation theory asserts that every deformation problem over a field of characteristic $0$ can be encoded by a differential graded Lie algebra (see \cite{lurie} and \cite{pridham}). More precisely, any deformation problem can be described by a solution of the \textit{Maurer-Cartan equation}:
$$d(x)+\frac{1}{2}[x,x]=0,$$

\noindent in some differential graded Lie algebra. The group obtained by the integration of the differential graded Lie algebra into a Lie group, called the \textit{gauge group}, moreover acts on the Maurer-Cartan set. The orbits of this action give isomorphism classes of deformation problems.

In \cite{dotsenko}, Dotsenko-Shadrin-Vallette showed that if the differential graded Lie algebra comes from a differential graded pre-Lie algebra, then the Maurer-Cartan equation, the gauge group and its action on the Maurer-Cartan set can be described in terms of pre-Lie operations. A differential graded pre-Lie algebra is a vector space $L$ with a bilinear operation $\star:L\otimes L\longrightarrow L$ such that
$$(x\star y)\star z-x\star(y\star z)=(-1)^{|y||z|}((x\star z)\star y-x\star (z\star y)),$$

\noindent and which satisfies the Leibniz rule with respect to the differential. Every differential graded pre-Lie algebra is in particular a differential graded Lie algebra with the graded commutator:
$$[x,y]=x\star y-(-1)^{|x||y|}y\star x.$$

\noindent Dotsenko-Shadrin-Vallette showed in particular that given a pre-Lie algebra $L$, the pre-Lie exponential map $exp:L^0\longrightarrow (1+L^0)$ induces an isomorphism between the gauge group and the group $(1+L^0,\circledcirc,1)$ with $\circledcirc$ the circular product defined by
$$a\circledcirc(1+b)=\sum_{n\geq 0}\frac{1}{n!}a\{\underbrace{b,\ldots,b}_n \},$$

\noindent where $-\{-,\ldots,-\}$ denotes the \textit{symmetric braces} determined by the pre-Lie structure $\star$, starting with $x\{y\}=x\star y$. Then, by writing the Maurer-Cartan equation as a zero-square equation, they prove that the action of the gauge group on the Maurer-Cartan set can be computed in terms of the circular product $\circledcirc$ as
$$e^{\lambda}\cdot \alpha=(e^{\lambda}\star\alpha)\circledcirc e^{-\lambda},$$

\noindent allowing us to have an easier way to compute the Deligne groupoid associated to any differential graded pre-Lie algebra over a field of characteristic $0$.\\

The aim of this paper is to develop a deformation theory in positive characteristic which generalizes the deformation theory controlled by pre-Lie algebras over a field of characteristic $0$ developed in \cite{dotsenko}. Our idea is to use differential graded pre-Lie algebras with divided powers.\\

The notion of a pre-Lie algebra with divided powers (or $\Gamma(\mathcal{P}re\mathcal{L}ie,-)$-algebra) has been studied by Cesaro in \cite{cesaro}. He showed in particular that every pre-Lie algebra with divided powers comes equipped with weighted brace operations $-\{-,\ldots,-\}_{r_1,\ldots,r_n}$, for each collection of integers $r_1,\ldots,r_n\geq 0$, which satisfy similar identities as the quantities
$$x\{y_{1},\ldots,y_{n}\}_{r_{1},\ldots,r_{n}}=\frac{1}{\prod_i r_i!}x\{\underbrace{y_1,\ldots,y_1}_{r_1},\ldots,\underbrace{y_n,\ldots,y_n}_{r_n}\}$$

\noindent in a pre-Lie algebra over a field of characteristic $0$ (see \cite[Propositions 5.9-5.10]{cesaro} for a precise list of these identities).\\

Every differential graded pre-Lie algebra with divided powers $L$ comes equipped with analogous weighted brace operations $-\{-,\ldots,-\}_{r_1,\ldots,r_n}$ which satisfy a graded version of the identities satisfied by weighted braces in the non graded framework. In this context, we have an analogue of the Maurer-Cartan equation:
$$d(x)+x\{x\}_{1}=0.$$

\noindent With suitable convergence hypothesis, we also get that the circular product $\circledcirc$ can be written as
$$a\circledcirc(1+b)=\sum_{n\geq 0}a\{b\}_{n}$$

\noindent and gives rise to a group structure on $1+L^0$. This group is called the \textit{gauge group} of $L$. As in characteristic $0$, we also show that this group acts on the Maurer-Cartan set of $L$.

\begin{thmA}\label{theoremA}
    Let $\mathbb{K}$ be a ring.
    \begin{enumerate}[(i)]
        \item In any complete differential graded pre-Lie algebra with divided powers $L$, the circular product $\circledcirc$, defined as above, endows the set $1+L^0$ with a group structure.
        \item If we denote by $d$ the differential of $L$, then this group acts on the Maurer-Cartan set via the formula
    $$(1+\mu)\cdot \alpha=(\alpha+\mu\{\alpha\}_1-d(\mu))\circledcirc(1+\mu)^{\circledcirc-1}.$$
\end{enumerate}
\end{thmA}

We prove that this new deformation theory satisfies an analogue of the Goldman-Millson theorem given in \cite[$\mathsection$2.4]{goldman}. Let $\text{\normalfont Deligne}(L,A)$ be the Deligne groupoid of the complete dg pre-Lie algebra with divided powers $L\otimes\mathfrak{m}_A$, where $L$ is a dg pre-Lie algebra with divided powers and $\mathfrak{m}_A$ the maximal ideal of a local artinian algebra $A$ over the field of fraction $\textbf{K}$ of $\mathbb{K}$. We precisely get the following result.

\begin{thmA}\label{theoremB}
    Let $\mathbb{K}$ be a noetherian integral domain and $\textbf{K}$ its field of fractions. Let $L$ and $\overline{L}$ be two positively graded $\Gamma(\mathcal{P}re\mathcal{L}ie,-)$-algebras. Let $\varphi:L\longrightarrow\overline{L}$ be a morphism of $\Gamma(\mathcal{P}re\mathcal{L}ie,-)$-algebras such that $H^0(\varphi)$ and $H^1(\varphi)$ are isomorphisms and $H^2(\varphi)$ is a monomorphism. Then for every local artinian $\textbf{K}$-algebra $A$, the induced functor $\varphi_*:\text{\normalfont Deligne}(L,A)\longrightarrow\text{\normalfont Deligne}(\overline{L},A)$ is an equivalence of groupoids.
\end{thmA}

Other approaches to generalize the usual deformation theory in the positive characteristic framework have been proposed recently in the literature. We have for instance a deformation theory in an associative context, via $\mathcal{A}_{\infty}$-algebras, which is used to study deformations of group representations (see \cite{associative}). Another approach is given by (spectral) partition Lie algebras to get a full generalization of the Lurie-Pridham correspondence in the setting of a field with positive characteristic (see \cite{pdoperads,partitionliealgebra}).\\

The main motivation for the approach developed in this paper is that operadic deformation problems are expressed in terms of pre-Lie structures. The goal is then to compute the $\pi_{0}$ of a mapping space $\text{\normalfont Map}(B^c(\mathcal{C}),\mathcal{P})$, where we take any dg operad $\mathcal{P}$ on the target and the operad $B^c(\mathcal{C})$ given by the cobar of a dg coaugmented cooperad $\mathcal{C}$ on the source. Recall simply that $B^c(\mathcal{C})$ defines a cofibrant operad when $\mathcal{C}$ is cofibrant as a symmetric sequence ($\Sigma_\ast$-cofibrant). It is well known that, over a field of characteristic $0$, the $\pi_0$ of this mapping space is the set of isomorphism classes of the Deligne groupoid of the Lie algebra $\text{\normalfont Hom}_\Sigma(\overline{\mathcal{C}},\overline{\mathcal{P}})$. Using the pre-Lie algebra structure of $\text{\normalfont Hom}_\Sigma(\overline{\mathcal{C}},\overline{\mathcal{P}})$, this can be seen as a consequence of the computations in \cite{dotsenko}. To extend this result, we first show that $\text{\normalfont Hom}_\Sigma(\overline{\mathcal{C}},\overline{\mathcal{P}})$ admits a structure of dg pre-Lie algebra with divided powers. Then we get the following statement.

\begin{thmA}\label{theoremC}
Let $\mathbb{K}$ be a field. Suppose that $\mathcal{C}$ is a $\Sigma_{*}$-cofibrant coaugmented dg cooperad and $\mathcal{P}$ an augmented dg operad. We then have an isomorphism:
$$\pi_{0}(\text{\normalfont Map}(B^c(\mathcal{C}),\mathcal{P}))\simeq \pi_{0}\text{\normalfont Deligne}(\text{\normalfont Hom}_{\Sigma}(\overline{\mathcal{C}},\overline{\mathcal{P}})),$$

\noindent where $\pi_{0}\text{\normalfont Deligne}(\text{\normalfont Hom}_{\Sigma}(\overline{\mathcal{C}},\overline{\mathcal{P}}))$ denotes the set of isomorphism classes of the Deligne groupoid.
\end{thmA}

This theorem gives a first step for the calculation of the homotopy groups of a mapping space $\text{\normalfont Map}(B^c (\mathcal{C}),\mathcal{P})$ over any field.\\

In the first part of this paper, we recall some definitions and properties on pre-Lie algebras and pre-Lie algebras with divided powers: in $\mathsection$\ref{sec:11} we briefly review the definition of the notion of a pre-Lie algebra and the construction of the corresponding operad; in $\mathsection$\ref{sec:12}, we review the definition of a pre-Lie algebra with divided powers and of the weighted brace operations.\\ \indent In the second part, we develop the deformation theory for differential graded pre-Lie algebras with divided powers: in $\mathsection$\ref{sec:21}, we study pre-Lie algebras with divided powers in the dg framework; in $\mathsection$\ref{sec:22}, we define the circular product and prove assertion $(i)$ of Theorem \ref{theoremA}; in $\mathsection$\ref{sec:23}, we define the Maurer-Cartan set and prove assertion $(ii)$ of Theorem \ref{theoremA}; in $\mathsection$\ref{sec:24}, we finally prove our analogue of the Goldman-Millson theorem (Theorem \ref{theoremB}) for this new deformation theory.\\ \indent We conclude this article with our application of this deformation theory for operadic deformation problems: in $\mathsection$\ref{sec:31}, we introduce some basic definitions on symmetric sequences and operads which will be useful to write our formulas; in $\mathsection$\ref{sec:32}, we study the structure of differential graded pre-Lie algebra with divided powers of the convolution operad; in $\mathsection$\ref{sec:33}, we finally give a proof of Theorem \ref{theoremC}.

\subsubsection*{Conventions}

We denote the symmetric group on $n$ letters by $\Sigma_n$. Recall that a permutation $\sigma\in\Sigma_{r_1+\cdots+r_n}$ is a \textit{shuffle permutation of type $(r_1,\ldots,r_n)$} if $\sigma$ preserves the order on the subsets $\{r_1+\cdots+r_i+1<\cdots<r_1+\cdots+r_{i+1}\}$ of $\{1<\cdots<r_1+\cdots+r_n\}$. The shuffle permutation $\sigma$ is \textit{pointed} if we also have $\sigma(1)<\sigma(r_1+1)<\cdots<\sigma(r_1+\cdots+r_{n-1}+1)$.\\ We denote by $Sh(r_1,\ldots,r_n)$ the subset of $\Sigma_{r_1+\cdots+r_n}$ composed of shuffle permutations of type $(r_1,\ldots,r_n)$ and by $Sh_\ast(r_1,\ldots,r_n)$ the subset of $Sh(r_1,\ldots,r_n)$ composed of pointed shuffle permutations.

\subsubsection*{Acknowledgements}

I acknowledge support from the Labex CEMPI (ANR-11-LABX-0007-01) and from the FNS-ANR project OCHoTop (ANR-18CE93-0002-01).\\ I am grateful to Benoit Fresse for his advice and support throughout the writing of this article, as well as to the anonymous referee for her/his comments and suggestions, which greatly contributed to the improvement of this work.

\section{Recollections on pre-Lie algebras with divided powers}



We first recall some definitions and basic properties on pre-Lie algebras and pre-Lie algebras with divided powers. Pre-Lie algebras were introduced in deformation theory by Gerstenhaber in \cite{gerstenhaber}, while pre-Lie algebras with divided powers were introduced by Cesaro in \cite{cesaro}.\\

In $\mathsection$\ref{sec:11}, we give brief recollections on the notion of a pre-Lie algebra. We will more particularly see pre-Lie algebras as algebras over an operad introduced by Chapoton-Livernet in \cite{chapoton}, the rooted tree operad, of which we also recall the definition in this subsection.\\

In $\mathsection$\ref{sec:12}, we give recollections on the notion of a pre-Lie algebra with divided powers. These objects can be seen as pre-Lie algebras with some extra operations. We will focus on some of these operations called weighted braces that will mimic the quantities which appear in the definition of the circular product.

\subsection{Pre-Lie algebras and the rooted tree operad}\label{sec:11}

We will use the following basic definitions.

\begin{defi}\label{prelie}
    A {\normalfont pre-Lie algebra} over a ring $\mathbb{K}$ is a $\mathbb{K}$-module $L$ endowed with a bilinear morphism $\star:L\otimes L\longrightarrow L$ such that
    $$(x\star y)\star z-x\star (y\star z)=(x\star z)\star y-x\star (z\star y).$$
\end{defi}

The category of pre-Lie algebras is isomorphic to the category of symmetric braces algebras (see \cite{oudom} or \cite{lada}). The symmetric braces $-\{-,\ldots,-\}$ are defined by induction on the length of the brace by

$\begin{array}{rrl}
    a\{\} & = & a,\\
    a\{b_{1}\} & = & a\star b_{1},\\
    \forall n\geq 1, a\{b_{1},\ldots,b_{n}\} & = & \displaystyle a\{b_{1},\ldots,b_{n-1}\}\{b_{n}\}-\sum_{i=1}^{n-1}a\{b_{1},\ldots,b_{i-1},b_{i}\{b_{n}\}, b_{i+1},\ldots,b_{n-1}\},
\end{array}$

\noindent for all $a,b_{1},\ldots,b_{n}\in L$.\\

For our purpose, it will be more convenient to see pre-Lie algebras as algebras over an operad. This operad can be described in terms of rooted trees as follows.

\begin{defi} {\normalfont (see \cite[$\mathsection$1.5]{chapoton})}
    We call {\normalfont $n$-rooted tree} a non-planar tree with $n$ vertices equipped with a numbering from $1$ to $n$, together with a distinguished vertex called the {\normalfont root}. By convention, we choose to put the root at the bottom in any representation of a tree.

We let $\mathcal{RT}(n)$ to be the set of all trees with $n$ vertices, and $\mathcal{P}re\mathcal{L}ie(n)=\mathbb{K}[\mathcal{RT}(n)]$.
\end{defi}

The collection $\mathcal{P}re\mathcal{L}ie$ is endowed with an operad structure. The action of $\Sigma_{n}$ on $\mathcal{P}re\mathcal{L}ie(n)$ for all $n\geq 1$ is given by the permutation of the indices attached to the vertices. The $i$-th partial composition $S\circ_{i} T\in\mathcal{P}re\mathcal{L}ie(p+q-1)$ of $S\in \mathcal{RT}(p)$ and $T\in \mathcal{RT}(q)$ is given by the sum of all the possible trees obtained by putting $T$ in the vertex $i$ of $S$, with the obvious choice of the numbering (see an example in \cite[$\mathsection$1.5]{chapoton}). This operad is also called the \textit{rooted tree operad}.\\

One can show that the algebras over the rooted tree operad are precisely the pre-Lie algebras (see \cite[$\mathsection$1.9]{chapoton}). In particular, the symmetric braces are given by the trees $F_n$ for $n\geq 0$ called \textit{corollas with $n$ leaves}:

\begin{equation*}
    F_{n}=
    \begin{tikzpicture}[baseline={([yshift=-.5ex]current bounding box.center)},scale=0.6]
    \node[draw,circle,scale=0.6] (i) at (0,0) {$1$};
    \node[draw,circle,scale=0.6] (1) at (-1,1) {$2$};
    \node[draw,circle,scale=0.6] (2) at (0,1) {$3$};
    \node[draw,circle,scale=0.6] (3) at (2,1) {$n+1$};
    \node (a) at (0.9,1) {$\cdots$};
    \draw (2) -- (i);
    \draw (i) -- (1);
    \draw (i) -- (3);
    \end{tikzpicture}.\ \
\end{equation*}

\subsection{Pre-Lie algebras with divided powers}\label{sec:12}

In this part, we recall the notion of a pre-Lie algebra with divided powers. We obtain this definition as a particular case of a general construction, for algebras over an operad, which we briefly recall.\\

Every operad $\mathcal{P}$ on a suitable monoidal category $C$ gives a functor $\mathcal{S}(\mathcal{P},-):C\longrightarrow C$, called the \textit{Schur functor}, defined by
$$\mathcal{S}(\mathcal{P},V)=\bigoplus_{n\geq 0} \mathcal{P}(n)\otimes_{\Sigma_n}V^{\otimes n},$$

\noindent where we consider, in the direct sum, the coinvariants of $\mathcal{P}(n)\otimes V^{\otimes n}$ under the diagonal action of $\Sigma_n$ given by its action on $\mathcal{P}(n)$ and its action by permutation on the tensor product $V^{\otimes n}$. The image of a tensor product $p\otimes v_1\otimes \cdots\otimes v_n\in\mathcal{P}(n)\otimes V^{\otimes n}$ in $\mathcal{P}(n)\otimes_{\Sigma_n}V^{\otimes n}$ will be denoted by $p(v_1,\ldots,v_n)$. The Schur functor defines a monad and the category of algebras over this monad is the usual category of algebras over the operad $\mathcal{P}$. In particular, pre-Lie algebras in the sense of Definition \ref{prelie} are identified with $\mathcal{S}(\mathcal{P}re\mathcal{L}ie,-)$-algebras.\\

In the above definition, one can chose to take invariants instead of coinvariants. We obtain a new functor $\Gamma(\mathcal{P},-):C\longrightarrow C$ defined by
$$\Gamma(\mathcal{P},V)=\bigoplus_{n\geq 0}\mathcal{P}(n)\otimes^{\Sigma_n}V^{\otimes n}.$$

\noindent If $\mathcal{P}(0)=0$, this functor also gives a monad (see \cite[$\mathsection$1.1.18]{fressebis}). The algebras over this monad are called \textit{$\mathcal{P}$-algebras with divided powers}. The motivation for this terminology comes from the fact that $\Gamma(\mathcal{C}om,-)$-algebras are precisely the usual commutative and associative algebras over $\mathbb{K}$ with divided powers.\\

Note that if $C$ is a category whose objects are $\mathbb{K}$-modules, then we have a morphism of monads $Tr:\mathcal{S}(\mathcal{P},V)\longrightarrow\Gamma(\mathcal{P},V)$ called the \textit{trace map} and defined by
$$Tr(p(v_1,\ldots,v_n))=\sum_{\sigma\in\Sigma_n}(\sigma\cdot p) \otimes v_{\sigma^{-1}(1)}\otimes\cdots\otimes v_{\sigma^{-1}(n)}.$$

\noindent 
If $\mathbb{K}$ is a field of characteristic $0$, the trace map is an isomorphism. It is no longer the case in general when $char(\mathbb{K})\neq 0$.\\

In the case $C=\text{\normalfont Mod}_\mathbb{K}$ of the category of $\mathbb{K}$-modules and $\mathcal{P}=\mathcal{P}re\mathcal{L}ie$, if $V$ is a free $\mathbb{K}$-module, we however have an isomorphism of modules given by the \textit{orbit morphism} $\mathcal{O}:\mathcal{S}(\mathcal{P}re\mathcal{L}ie,V)\longrightarrow\Gamma(\mathcal{P}re\mathcal{L}ie,V)$ defined as follows. Let $n\geq 1$ and $\mathfrak{t}\in\mathcal{P}re\mathcal{L}ie(n)\otimes V^{\otimes n}$ be a basis element. We set
$$\mathcal{O}(t)=\sum_{\sigma\in\Sigma_n/\text{\normalfont Stab}_{\Sigma_n}(t)}\sigma\cdot \mathfrak{t},$$

\noindent where ${\normalfont \text{Stab}}_{\Sigma_n}(\mathfrak{t})$ is the stabilizer of $\mathfrak{t}$ under the diagonal action of $\Sigma_{n}$ on $\mathcal{P}re\mathcal{L}ie(n)\otimes V^{\otimes n}$. The map $\mathcal{O}$ is then extended by linearity on $\mathcal{P}re\mathcal{L}ie(n)\otimes V^{\otimes n}$.



\begin{thm}\label{Cesaro}\text{\normalfont (A. Cesaro, \cite{cesaro})}
    Every pre-Lie algebra with divided powers $L$ comes equipped with operations $-\{-,\ldots,-\}_{r_1,\ldots,r_n}:L^{\times n+1}\longrightarrow L$ called \text{\normalfont weighted braces} which satisfy the following identities:
\begin{enumerate}[(i)]
    \item $x\{y_{\sigma(1)},\ldots,y_{\sigma(n)}\}_{r_{\sigma(1)},\ldots,r_{\sigma(n)}}=x\{y_{1},\ldots,y_{n}\}_{r_{1},\ldots,r_{n}},$
    \item $x\{y_{1},\ldots,y_{i-1},y_{i},y_{i+1},\ldots,y_{n}\}_{r_{1},\ldots,r_{i-1},0,r_{i+1},\ldots,r_{n}}=x\{y_{1},\ldots,y_{i-1},y_{i+1},\ldots,y_{n}\}_{r_{1},\ldots,r_{i-1},r_{i+1},\ldots,r_{n}},$
    \item $x\{y_{1},\ldots,\lambda y_{i},\ldots,y_{n}\}_{r_{1},\ldots,r_{i},\ldots,r_{n}}=\lambda^{r_{i}}x\{y_{1},\ldots,y_{i},\ldots,y_{n}\}_{r_{1},\ldots,r_{i},\ldots,r_{n}},$
    \item $\displaystyle x\{y_{1},\ldots,y_{i},y_{i},\ldots,y_{n}\}_{r_{1},\ldots,r_{i},r_{i+1},\ldots,r_{n}}=\binom{r_{i}+r_{i+1}}{r_{i}}x\{y_{1},\ldots,y_{i},\ldots,y_{n}\}_{r_{1},\ldots,r_{i-1},r_{i}+r_{i+1},r_{i+2},\ldots,r_{n}},$
    \item $\displaystyle x\{y_{1},\ldots,y_{i}+\widetilde{y_{i}},\ldots,y_{n}\}_{r_{1},\ldots,r_{i},\ldots,r_{n}}=\sum_{s=0}^{r_{i}}x\{y_{1},\ldots,y_{i},\widetilde{y_{i}},\ldots,y_{n}\}_{r_{1},\ldots,s,r_{i}-s,\ldots,r_{n}},$
    \item $x\{y_{1},\ldots,y_{n}\}_{r_{1},\ldots,r_{n}}\{z_{1},\ldots,z_{m}\}_{s_{1},\ldots,s_{m}}=$
$$\sum_{s_{i}=\beta_{i}+\sum\alpha_{i}^{\bullet,\bullet}}\frac{1}{\prod_{j}(r_{j})!}x\{y_{1}\{z_{1},\ldots,z_{m}\}_{\alpha_{1}^{1,1},\ldots,\alpha_{m}^{1,1}},\ldots,y_{1}\{z_{1},\ldots,z_{m}\}_{\alpha_{1}^{1,r_{1}},\ldots,\alpha_{m}^{1,r_{1}}},$$
$$\ldots,y_{n}\{z_{1},\ldots,z_{m}\}_{\alpha_{1}^{n,1},\ldots,\alpha_{m}^{n,1}},\ldots,y_{n}\{z_{1},\ldots,z_{m}\}_{\alpha_{1}^{n,r_{n}},\ldots,\alpha_{m}^{n,r_{n}}},z_{1},\ldots,z_{m}\}_{1,\ldots,1,\beta_{1},\ldots,\beta_{m}},$$
\end{enumerate}

\noindent for all $n,m\geq 0$, $r_{1},\ldots,r_{n},s_{1},\ldots,s_{m}\geq 0$, $1\leq i\leq n$, $\sigma\in\Sigma_{n}$, $\lambda\in\mathbb{K}$ and $x,y_{1},\ldots,y_{n},z_{1},\ldots,z_{m}\in L$.
\end{thm}

Note that the formula $(vi)$ is written in a form that uses fractions for more convenience, but can be reduced to $\mathbb{Z}$ using the other formulas. The process works as follows. Let $i$ such that $1\leq i\leq n$. In the sum, we first fix $\beta_{1},\ldots,\beta_{m}$ and $\alpha_{j}^{p,q}$ for $1\leq j\leq m$, $1\leq q\leq r_{j}$ and $p\neq i$. We obtain a sum with $(\alpha_{1}^{i,1},\ldots,\alpha_{m}^{i,1},\ldots,\alpha_{1}^{i,r_{i}},\ldots,\alpha_{m}^{i,r_{i}})$ as variables. We identify this last tuple with a tuple of tuples of the form $((\alpha_{1}^{i,1},\ldots,\alpha_{m}^{i,1});\ldots;(\alpha_{1}^{i,r_{i}},\ldots,\alpha_{m}^{i,r_{i}}))$. Let $u$ be one of these tuples and suppose $u=(\underbrace{\widetilde{u_{1}},\ldots,\widetilde{u_{1}}}_{t_{1}},\ldots,\underbrace{\widetilde{u_{q}},\ldots,\widetilde{u_{q}}}_{t_{q}})$ up to permutation. Note that, if $\widetilde{u_{1}},\ldots,\widetilde{u_{q}}$ are given, we exactly have $\displaystyle\frac{r_{i}!}{t_{1}!\cdots t_{q}!}$ such terms occurring in the sum. Then, by using the symmetry formula $(i)$, the formula $(iv)$ and by summing over all such tuples, we have in the sum:
$$\frac{1}{\prod_{j}(r_{j})!}\frac{r_{i}!}{t_{1}!\cdots t_{q}!}t_{1}!\cdots t_{q}!\ x\{y_{1}\{z_{1},\ldots,z_{m}\}_{\alpha_{1}^{1,1},\ldots,\alpha_{m}^{1,1}},\ldots,y_{1}\{z_{1},\ldots,z_{m}\}_{\alpha_{1}^{1,r_{1}},\ldots,\alpha_{m}^{1,r_{1}}},\ldots,$$
$$y_{i}\{z_{1},\ldots,z_{m}\}_{\widetilde{u_{1}}},\ldots,y_{i}\{z_{1},\ldots,z_{m}\}_{\widetilde{u_{q}}},\ldots$$
$$\ldots,y_{n}\{z_{1},\ldots,z_{m}\}_{\alpha_{1}^{n,1},\ldots,\alpha_{m}^{n,1}},\ldots,y_{n}\{z_{1},\ldots,z_{m}\}_{\alpha_{1}^{n,r_{n}},...,\alpha_{m}^{n,r_{n}}},z_{1},\ldots,z_{m}\}_{1,\ldots,t_{1},\ldots,t_{q},\ldots,1,\beta_{1},\ldots,\beta_{m}},$$

\noindent where we have set $y_{i}\{z_{1},\ldots,z_{m}\}_{\widetilde{u_{k}}}=y_{i}\{z_{1},\ldots,z_{m}\}_{\alpha_{1},\ldots,\alpha_{m}}$ if $\widetilde{u_{k}}=(\alpha_{1},\ldots,\alpha_{m})$. Hence, it gives:
$$\frac{1}{\prod_{j\neq i}(r_{j})!}x\{y_{1}\{z_{1},\ldots,z_{m}\}_{\alpha_{1}^{1,1},\ldots,\alpha_{m}^{1,1}},\ldots,y_{1}\{z_{1},\ldots,z_{m}\}_{\alpha_{1}^{1,r_{1}},\ldots,\alpha_{m}^{1,r_{1}}},\ldots,$$
$$y_{i}\{z_{1},\ldots,z_{m}\}_{\widetilde{u_{1}}},\ldots,y_{i}\{z_{1},\ldots,z_{m}\}_{\widetilde{u_{q}}},\ldots$$
$$\ldots,y_{n}\{z_{1},\ldots,z_{m}\}_{\alpha_{1}^{n,1},\ldots,\alpha_{m}^{n,1}},\ldots,y_{n}\{z_{1},\ldots,z_{m}\}_{\alpha_{1}^{n,r_{n}},\ldots,\alpha_{m}^{n,r_{n}}},z_{1},\ldots,z_{m}\}_{1,\ldots,t_{1},v,t_{q},\ldots,1,\beta_{1},\ldots,\beta_{m}}.$$

By iterating this argument on the other terms, we obtain an expression over $\mathbb{Z}$.\\

The reader can find an example of such a reduction of the formula $(vi)$ in \cite[Example 5.11]{cesaro}, as well as a proof of the previous theorem (see \cite[Propositions 5.9-5.10]{cesaro}).\\

We give the explicit construction of the weighted braces.

\begin{cons}
    We regard the weighted braces $x\{y_1,\ldots,y_n\}_{r_1,\ldots,r_n}$ as the action of the corolla $F_{\sum_i r_i}$ on the tensor $x\otimes\underbrace{y_1\otimes\cdots\otimes y_1}_{r_1}\otimes\cdots\otimes\underbrace{y_n\otimes\cdots\otimes y_n}_{r_n}$ where we regard the $y_{i}$'s as distinct variables (see below). If $y_{i}\neq y_{j}$ for all $i\neq j$, then we precisely set
    $$x\{y_1,\ldots,y_n\}_{r_1,\ldots,r_n}=\gamma(\mathcal{O}F_{\sum_i r_i}(x,\underbrace{y_1,\ldots, y_1}_{r_1},\ldots,\underbrace{y_n,\ldots, y_n}_{r_n})),$$

    \noindent where $\gamma$ is the $\Gamma(\mathcal{P}re\mathcal{L}ie,-)$-algebra structure on $L$.\\

    In order to include the case where some of the $y_i$'s might be the same, let $E_n$ to be the free $\mathbb{K}$-module generated by a basis $e,e_1,\ldots,e_n$. Let $\psi_{x,y_1,\ldots,y_n}:E_n\longrightarrow L$ be the morphism which sends $e$ to $x$ and $e_i$ to $y_i$ for all $1\leq i\leq n$. We obtain a morphism $\Gamma(\mathcal{P}re\mathcal{L}ie,\psi_{x,y_1,\ldots,y_n}):\Gamma(\mathcal{P}re\mathcal{L}ie,E_n)\longrightarrow\Gamma(\mathcal{P}re\mathcal{L}ie,L)$. We then take the orbit map at the source and apply this morphism next to have a good definition of the weighted braces.
\end{cons}

\begin{remarque}
    The converse of the previous theorem is also true, provided that $L$ is free as a $\mathbb{K}$-module.
\end{remarque}


\section{Deformation theory of $\Gamma(\mathcal{P}re\mathcal{L}ie,-)$-algebras}

The main goal of this section is to extend the results proved by Dotsenko-Shadrin-Vallette in \cite{dotsenko} in the context of a ring of positive characteristic. The main idea is that formulas which define the circular product and the gauge action can be written in terms of weighted brace operations.\\

In $\mathsection$\ref{sec:21}, we revisit the definition of pre-Lie algebras with divided powers in the dg framework. In particular, we give the analogue of the weighted brace operations. We then make explicit an example of differential graded pre-Lie algebras with divided powers given by differential graded brace algebras.\\

In $\mathsection$\ref{sec:22}, we define the circular product $\circledcirc$ in terms of weighted brace operations that will generalize the one given in \cite{dotsenko}. We then show that this induces a group called the gauge group associated to the $\Gamma(\mathcal{P}re\mathcal{L}ie,-)$-algebra.\\

In $\mathsection$\ref{sec:23}, we define the Maurer-Cartan equation in a $\Gamma(\mathcal{P}re\mathcal{L}ie,-)$-algebra, and then the Maurer-Cartan set. We also see that the gauge group acts on the Maurer-Cartan set by a similar formula given in \cite{dotsenko}.\\

In $\mathsection$\ref{sec:24}, we finally motivate this new deformation theory with an analogue of the Goldman-Millson theorem. This theorem, in particular, has the advantage to be true on integers.

\subsection{Differential graded pre-Lie algebras with divided powers}\label{sec:21}

As we are dealing with differential graded modules, our first goal is to define and study differential graded pre-Lie algebras with divided powers.\\

In the following sections, we assume that dg modules are equipped with a cohomological grading convention. We will denote by $\otimes$ the usual tensor product of graded modules over any ring $\mathbb{K}$. This induces a symmetric monoidal category that we will denote by $\text{\normalfont dgMod}_\mathbb{K}$. If there is no possible confusion, then we will denote by $\pm$ any sign produced by the Koszul sign rule.

\subsubsection{Weighted braces on $\Gamma(\mathcal{P}re\mathcal{L}ie,-)$-algebras}

Our main goal here is to extend \cite[Proposition 5.13]{cesaro} in the context of dg modules. We first begin by a basic definition.

\begin{defi}
    A {\normalfont differential graded pre-Lie algebra} is an algebra over the monad $\mathcal{S}(\mathcal{P}re\mathcal{L}ie,-):\text{\normalfont dgMod}_\mathbb{K}\longrightarrow\text{\normalfont dgMod}_\mathbb{K}$.
\end{defi}

Equivalently, we can see that a differential graded pre-Lie algebra is a graded module $ L=\bigoplus_{k\in\mathbb{Z}}L^k$ endowed with a morphism of graded modules $\star:L\otimes L\longrightarrow L$ such that
    $$(x\star y)\star z-x\star(y\star z)=\pm ((x\star z)\star y-x\star (z\star y))$$

    \noindent and a differential $d:L^k\longrightarrow L^{k+1}$, which satisfies
    $$d(x\star y)=d(x)\star y\pm x\star d(y),$$

    \noindent where $\pm$ is the sign yielded by the permutation of $x$ and $d$.\\

    We now define the notion of a pre-Lie algebra with divided powers in the dg framework.

\begin{defi}
    A {\normalfont differential graded pre-Lie algebra with divided powers} is an algebra over the monad $\Gamma(\mathcal{P}re\mathcal{L}ie,-):\text{\normalfont dgMod}_\mathbb{K}\longrightarrow \text{\normalfont dgMod}_\mathbb{K}$.
\end{defi}


Let $L\in\text{\normalfont dgMod}_\mathbb{K}$. Suppose that $L^k$ is a free $\mathbb{K}$-module for every $k\in\mathbb{Z}$. Let $\mathcal{L}$ be a basis of $L$ composed of homogeneous elements. Then we have a basis on $\mathcal{P}re\mathcal{L}ie(n)\otimes L^{\otimes n}$ for every $n\geq 1$ given by tensors $T\otimes e_1\otimes\cdots\otimes e_n$ where $T\in \mathcal{RT}(n)$ and $e_1,\ldots,e_n\in\mathcal{L}$. We denote by $\mathcal{RT}(n)\otimes\mathcal{L}^{\otimes n}$ this basis. If $char(\mathbb{K})\neq 2$, then the action of $\Sigma_n$ on $\mathcal{P}re\mathcal{L}ie(n)\otimes L^{\otimes n}$ does not restrict to an action on $\mathcal{RT}(n)\otimes\mathcal{L}^{\otimes n}$ because of the Koszul sign rule. To handle things properly, we put this sign apart. We endow $\mathcal{RT}(n)\otimes\mathcal{L}^{\otimes n}$ with the diagonal action of $\Sigma_n$ where $\Sigma_n$ acts on $\mathcal{L}^{\otimes n}$ by the permutation of elements where we omit the Koszul sign. Given a tensor $\mathfrak{t}\in\mathcal{RT}(n)\otimes\mathcal{L}^{\otimes n}$, we denote by $X_\mathfrak{t}$ the orbit of $\mathfrak{t}$ under this action so that we have the following equality of graded $\mathbb{K}$-modules:
$$\mathcal{P}re\mathcal{L}ie(n)\otimes L^{\otimes n}=\bigoplus_{\overline{\mathfrak{t}}\in(\mathcal{RT}(n)\otimes\mathcal{L}^{\otimes n})/\Sigma_n}\mathbb{K}[X_\mathfrak{t}].$$

\noindent We recover the Koszul sign rule in the following way. Let $\mathfrak{t}\in\mathcal{RT}(n)\otimes\mathcal{L}^{\otimes n}$. For every $\sigma\in\Sigma_n$, we define $\varepsilon(\sigma,\mathfrak{t})\in\{\pm 1\}\subset\mathbb{K}^\times$ as the Koszul sign which appears after the usual action of $\sigma$ on $\mathfrak{t}$ in the graded module $\mathcal{P}re\mathcal{L}ie(n)\otimes L^{\otimes n}$. For every $\sigma,\tau\in\Sigma_n$, we have the identity
$$\varepsilon(\sigma\tau,\mathfrak{t})=\varepsilon(\sigma,\tau\cdot\mathfrak{t})\varepsilon(\tau,\mathfrak{t}).$$

\noindent Equivalently, we can see $\varepsilon$ as a functor from the groupoid with $X_\mathfrak{t}$ as set of objects and $\text{\normalfont Hom}(\mathfrak{t}',\mathfrak{t}'')=\{\sigma\in\Sigma_n\ |\ \sigma\cdot\mathfrak{t}'=\mathfrak{t}''\}$ for $\mathfrak{t}',\mathfrak{t}''\in X_\mathfrak{t}$ to the groupoid, denoted by $\{\pm 1\}$, with only one object $\ast$ and $\text{\normalfont Hom}(\ast,\ast)=\{\pm 1\}\subset\mathbb{K}^\times$. We can then define the $\Sigma_n$-representation $\mathbb{K}[X_\mathfrak{t}]^\pm$ as $\mathbb{K}[X_\mathfrak{t}]$ endowed with the action of $\Sigma_n$ defined by
$$\sigma\cdot x^\pm=\varepsilon(\sigma,x)(\sigma\cdot x)^\pm$$

\noindent for every $\sigma\in\Sigma_n$ and $x\in X_\mathfrak{t}$. We thus have the following identity of $\Sigma_n$-representations:
$$\mathcal{P}re\mathcal{L}ie(n)\otimes L^{\otimes n}=\bigoplus_{\overline{\mathfrak{t}}\in(\mathcal{RT}(n)\otimes\mathcal{L}^{\otimes n})/\Sigma_n}\mathbb{K}[X_\mathfrak{t}]^\pm.$$

\noindent Our purpose it to define an analogue of the orbit map, by using the above decomposition in set-theoretic orbits. We rely on the following lemma.

\begin{lm}\label{technique}
    Let $G$ be a group and $H\subset G$ be a subgroup. We consider the action of $G$ on $G/H$ by the left translation.  Let $X$ be the groupoid with as set of objects $G/H$ and with as morphisms $\text{\normalfont Hom}(x,x')=\{g\in G\ |\ g\cdot x=x'\}$. Let $\varepsilon:X\longrightarrow\{\pm1\}\subset\mathbb{K}^\times$ be a functor. We denote by $\varepsilon(g,x)$ the image of the morphism $g:x\longrightarrow g\cdot x$ under this functor. Consider the $G$-representation $\mathbb{K}[G/H]^\pm=\mathbb{K}[G/H]$ on which $G$ acts by $g\cdot x^\pm=\varepsilon(g,x)(g\cdot x)^\pm$ for every $x\in G/H$. For every $g\in G$, we denote by $\overline{g}$ its class in $G/H$ and by $[\overline{g}^\pm]$ the class of $\overline{g}^\pm\in\mathbb{K}[G/H]^\pm$ in $(\mathbb{K}[G/H]^\pm)_G$.
    \begin{itemize}
        \item If there exists $h\in H$ such that $\varepsilon(h,\overline{1})\neq 1$, then
        $$(\mathbb{K}[G/H]^\pm)_G=(\mathbb{K}/2\mathbb{K})[[\overline{1}^\pm]]\ ;\ (\mathbb{K}[G/H]^\pm)^G=\text{\normalfont Tor}_2(\mathbb{K})\left[\sum_{\overline{g}\in G/H}\varepsilon(g,\overline{1})\overline{g}^\pm\right],$$

        \noindent where $\text{\normalfont Tor}_2(\mathbb{K})$ denotes the set of $2$-torsion elements of $\mathbb{K}$.
        \item Otherwise, we have the identities
        $$(\mathbb{K}[G/H]^\pm)_G=\mathbb{K}[[\overline{1}^\pm]]\ ;\ (\mathbb{K}[G/H]^\pm)^G=\mathbb{K}\left[\sum_{\overline{g}\in G/H}\varepsilon(g,\overline{1})\overline{g}^\pm\right].$$
    \end{itemize}
\end{lm}

\begin{proof}
    We first compute $(\mathbb{K}[G/H]^\pm)_G$. For every $g\in G$, we have $\overline{g}^\pm=g\cdot(\varepsilon(g,\overline{1})\overline{1}^\pm)$ so that the $\mathbb{K}$-module $(\mathbb{K}[G/H]^\pm)_G$ is generated by $[\overline{1}^\pm]$. If there exists $h\in H$ such that $\varepsilon(h,\overline{1})\neq 1$ (meaning that $\varepsilon(h,\overline{1})=-1\neq 1$ in $\mathbb{K}$), then, for every $\lambda\in\mathbb{K}$, we have that $\lambda [\overline{1}^\pm]=\lambda[h\cdot\overline{1}^\pm]=-\lambda[\overline{1}^\pm]$ which shows that $(\mathbb{K}[G/H]^\pm)_G=(\mathbb{K}/2\mathbb{K})[[\overline{1}^\pm]]$. If $\varepsilon(H,\overline{1})\subset\{1\}$, then $(\mathbb{K}[G/H]^\pm)_G=\mathbb{K}[[\overline{1}^\pm]]$.\\

    We now compute $(\mathbb{K}[G/H]^\pm)^G$. Let $x=\sum_{\overline{g}\in G/H}\lambda_{\overline{g}}\overline{g}^\pm\in(\mathbb{K}[G/H]^\pm)^G$. For every $g\in G$, the identity $g\cdot x=x$ gives $\lambda_{\overline{g}}=\varepsilon(g,\overline{1})\lambda_{\overline{1}}$. If there exists $h\in H$ such that $\varepsilon(h,\overline{1})\neq 1$, then $\varepsilon(g,\overline{1})\lambda_{\overline{1}}=\lambda_{\overline{g}}=\lambda_{\overline{gh}}=\varepsilon(gh,\overline{1})\lambda_{\overline{1}}=\varepsilon(g,\overline{1})\varepsilon(h,\overline{1})\lambda_{\overline{1}}$ which gives $2\lambda_{\overline{1}}=0$. We then have $\lambda_{\overline{1}}\in\text{\normalfont Tor}_2(\mathbb{K})$ which shows that $(\mathbb{K}[G/H]^\pm)^G=\text{\normalfont Tor}_2(\mathbb{K})[\sum_{\overline{g}\in G/H}\varepsilon(g,\overline{1})\overline{g}^\pm]$. If $\varepsilon(H,\overline{1})\subset\{1\}$ then $(\mathbb{K}[G/H]^\pm)^G=\mathbb{K}[\sum_{\overline{g}\in G/H}\varepsilon(g,\overline{1})\overline{g}^\pm]$.
\end{proof}

\begin{lm}\label{technique2}
    In the situation of Lemma \ref{technique} and when $\varepsilon(H,\overline{1})\subset\{1\}$, we define the {\normalfont orbit map} $\mathcal{O}:(\mathbb{K}[G/H]^\pm)_G\longrightarrow (\mathbb{K}[G/H]^\pm)^G$ by
    $$\mathcal{O}([\overline{1}]^ \pm)=\sum_{\overline{g}\in G/H}\varepsilon(g,\overline{1})\overline{g}^\pm.$$

    \noindent This map is an isomorphism.
\end{lm}

\begin{proof}
    This is an immediate consequence of the second point of Lemma \ref{technique}.
\end{proof}

We can apply this lemma to our situation, by noting that for every $\mathfrak{t}\in\mathcal{RT}(n)\otimes\mathcal{L}^{\otimes n}$, the set $X_\mathfrak{t}$ is in bijection with $\Sigma_n/\text{\normalfont Stab}_{\Sigma_n}(\mathfrak{t})$. In order to apply Lemma \ref{technique2}, we need to remove some elements of $\mathcal{RT}(n)\otimes\mathcal{L}^{\otimes n}$. These elements are given by tensors $\mathfrak{t}\in\mathcal{RT}(n)\otimes\mathcal{L}^{\otimes n}$ such that there exists $\sigma\in\text{\normalfont Stab}_{\Sigma_n}(\mathfrak{t})$ with $\varepsilon(\sigma,\mathfrak{t})\neq 1$. We denote by $(\mathcal{RT}(n)\otimes\mathcal{L}^{\otimes n})^o$ the set of such elements and $(\mathcal{RT}(n)\otimes\mathcal{L}^{\otimes n})^r=(\mathcal{RT}(n)\otimes\mathcal{L}^{\otimes n})\setminus(\mathcal{RT}(n)\otimes\mathcal{L}^{\otimes n})^o$. In the case $char(\mathbb{K})=2$, we just have $(\mathcal{RT}(n)\otimes\mathcal{L}^{\otimes n})^r=\mathcal{RT}(n)\otimes\mathcal{L}^{\otimes n}$.\\

\noindent We note that these sets are stable under the action of $\Sigma_n$. Indeed, if $\mathfrak{t}\in\mathcal{RT}(n)\otimes\mathcal{L}^{\otimes n}$ is such that there exists $\sigma\in\text{\normalfont Stab}_{\Sigma_n}(\mathfrak{t})$ with $\varepsilon(\sigma,\mathfrak{t})\neq 1$, then for every $\tau\in\Sigma_n$, we have that $\tau\sigma\tau^{-1}\in\text{\normalfont Stab}_{\Sigma_n}(\tau\cdot\mathfrak{t})$ and 
$$\varepsilon(\tau\sigma\tau^{-1},\tau\cdot\mathfrak{t})=\varepsilon(\tau,\sigma\cdot\mathfrak{t})\varepsilon(\sigma,\mathfrak{t})\varepsilon(\tau^{-1},\tau\cdot\mathfrak{t})=\varepsilon(\tau,\mathfrak{t})\varepsilon(\sigma,\mathfrak{t})\varepsilon(\tau^{-1},\tau\cdot\mathfrak{t})=\varepsilon(\sigma,\mathfrak{t})\neq 1.$$

\noindent We deduce that the quotient $(\mathcal{RT}(n)\otimes\mathcal{L}^{\otimes n})/\Sigma_n$ is the disjoint union of $(\mathcal{RT}(n)\otimes\mathcal{L}^{\otimes n})^o/\Sigma_n$ and $(\mathcal{RT}(n)\otimes\mathcal{L}^{\otimes n})^r/\Sigma_n$. We then define
$$\mathcal{S}^r(\mathcal{P}re\mathcal{L}ie,L)=\bigoplus_{n\geq 1}(\mathbb{K}[(\mathcal{RT}(n)\otimes\mathcal{L}^{\otimes n})^r]^\pm)_{\Sigma_n}\subset\mathcal{S}(\mathcal{P}re\mathcal{L}ie,L).$$
We deduce the following proposition.

\begin{prop}\label{Oiso}
    If $\mathbb{K}$ is an integral domain and if $L^k$ is a free $\mathbb{K}$-module for every $k\in\mathbb{Z}$, then the map $\mathcal{O}:{\mathcal{S}}^r(\mathcal{P}re\mathcal{L}ie,L)\longrightarrow\Gamma(\mathcal{P}re\mathcal{L}ie,L)$ is an isomorphism.
\end{prop}

\begin{proof}
    Let $\mathcal{L}$ be a basis of $L$ composed of homogeneous elements. We adopt the same notations as before Lemma \ref{technique} and before the statement of Proposition \ref{Oiso}. We then note that we have
    $$\mathcal{S}^r(\mathcal{P}re\mathcal{L}ie,L)=\bigoplus_{n\geq 1}\bigoplus_{\overline{\mathfrak{t}}\in(\mathcal{RT}(n)\otimes\mathcal{L}^{\otimes n})^r/\Sigma_n}(\mathbb{K}[X_\mathfrak{t}]^\pm)_{\Sigma_n};$$
    $$\Gamma(\mathcal{P}re\mathcal{L}ie,L)=\bigoplus_{n\geq 1}\bigoplus_{\overline{\mathfrak{t}}\in(\mathcal{RT}(n)\otimes\mathcal{L}^{\otimes n})/{\Sigma_n}}(\mathbb{K}[X_\mathfrak{t}]^\pm)^{\Sigma_n}.$$
    
    \noindent If $char(\mathbb{K})=2$, then $2\mathbb{K}=0$ and $\text{\normalfont Tor}_2(\mathbb{K})=\mathbb{K}$ so that the proposition is an immediate consequence of Lemma \ref{technique2}.\\ If $char(\mathbb{K})\neq 2$, then by the first point of Lemma \ref{technique} and by noting that $\text{\normalfont Tor}_2(\mathbb{K})=0$ because $\mathbb{K}$ is an integral domain, we have that $(\mathbb{K}[(\mathcal{RT}(n)\otimes\mathcal{L}^{\otimes n})^o]^\pm)^{\Sigma_n}=0$. We are then reduced to analyze the orbit maps $\mathcal{O}:(\mathbb{K}[X_\mathfrak{t}]^\pm)_{\Sigma_n}\longrightarrow(\mathbb{K}[X_\mathfrak{t}]^\pm)^{\Sigma_n}$ for $\mathfrak{t}\in(\mathcal{RT}(n)\otimes\mathcal{L}^{\otimes n})^r$, which are isomorphisms by the second point of Lemma \ref{technique2}.\\

    In any case, we then obtain that $\mathcal{O}:\mathcal{S}^r(\mathcal{P}re\mathcal{L}ie,L)\longrightarrow\Gamma(\mathcal{P}re\mathcal{L}ie,L)$ is an isomorphism.
\end{proof}

\begin{thm}\label{prelie1}
    Let $\mathbb{K}$ be a ring. A graded pre-Lie algebra with divided powers $ L=\bigoplus_{k\in\mathbb{Z}}L^{k}$ over $\mathbb{K}$ comes equipped with operations, called {\normalfont{weighted braces}}, which have the following form.

    \begin{itemize}
        \item[-] If $char(\mathbb{K})=2$, weighted braces are maps
        $$-\{-,\ldots,-\}_{r_{1},\ldots,r_{n}}:L^{\times n+1}\longrightarrow L,$$
        \noindent defined for any collection of integers $r_{1},\ldots,r_{n}\geq 0$, which satisfy all formulas of Theorem \ref{Cesaro} and preserve the grading in the sense that
        $$L^{k}\{L^{k_{1}},\ldots,L^{k_{n}}\}_{r_{1},\ldots,r_{n}}\subset L^{k+k_{1}r_{1}+\cdots+k_{n}r_{n}}.$$

        \item[-] If $char(\mathbb{K})\neq 2$, by setting $ L^{ev}=\bigoplus_{k\in\mathbb{Z}} L^{2k}$ and $ L^{odd}=\bigoplus_{k\in\mathbb{Z}} L^{2k+1}$, weighted braces are maps
        $$-\{\underbrace{-,\ldots,-}_{p},\underbrace{-,\ldots,-}_{q}\}_{r_{1},\ldots,r_{p},1,\ldots,1}:L\times (L^{ev})^{\times p}\times (L^{odd})^{\times q}\longrightarrow L,$$
        \noindent defined for any collection of integers $p,q,r_{1},\ldots,r_{p}\geq 0$, which satisfy all formulas of Theorem \ref{Cesaro} with a sign given by the Koszul sign rule (see the Remark \ref{remsign} below) and preserve the grading.
    \end{itemize}

    Conversely, if $\mathbb{K}$ is an integral domain, if $L^k$ is a free $\mathbb{K}$-module for every $k\in\mathbb{Z}$ and if $L$ admits weighted brace operations, then $L$ is a $\Gamma(\mathcal{P}re\mathcal{L}ie,-)$-algebra.
\end{thm}

\begin{remarque}\label{remsign}
    If $char(\mathbb{K})\neq 2$, formulas $(i)$ and $(vi)$ of Theorem \ref{Cesaro} differ by a sign given by the Koszul sign rule. The sign which appears in formula $(i)$ is given by 
    $$y_{\sigma(1)}^{\otimes r_{\sigma(1)}}\otimes\cdots\otimes y_{\sigma(n)}^{\otimes r_{\sigma(n)}}\longmapsto\pm y_1^{\otimes r_1}\otimes\cdots\otimes y_n^{\otimes r_n}.$$
    
    \noindent In formula $(vi)$, for each $\beta_i$'s and $\alpha_i^{\bullet,\bullet}$'s, the sign which appears in the relevant term is given by
    $$z_1^{\otimes s_1}\otimes\cdots\otimes z_m^{\otimes s_m}\longmapsto\pm z_1^{\otimes\alpha_1^{1,1}}\otimes\cdots\otimes z_m^{\otimes\alpha_m^{1,1}}\otimes\cdots\otimes z_1^{\otimes\alpha_1^{n,1}}\otimes\cdots\otimes z_m^{\otimes\alpha_m^{n,r_n}}\otimes z_1^{\otimes\beta_1}\otimes\cdots\otimes z_m^{\otimes\beta_m}.$$

    \noindent Note that these signs are induced by the permutation of the odd degree elements between them. Since their associated weight are equal to $1$, formula $(vi)$ can still be written without rational coefficients by the same process as in the paragraph after Theorem \ref{Cesaro}.
\end{remarque}

In order to handle both of the cases, in the following, when taking elements with associated weights, we will tacitly suppose that if $char(\mathbb{K})\neq 2$, then all odd degree elements will have an associated weight equal to $1$.\\

\begin{proof}
    We basically do the same thing as in \cite[Proposition 5.10]{cesaro}. Let $x,y_{1},\ldots,y_{n}\in L$. Let $E_{x,y_1,\ldots,y_n}$ be the graded $\mathbb{K}$-module generated by $e, e_{1},\ldots,e_{n}$ with matching degrees. We have a morphism of graded modules from $\psi_{x,y_1,\ldots,y_n}:E_{x,y_1,\ldots,y_n}\longrightarrow L$ which sends $e$ to $x$ and $e_{i}$ to $y_{i}$ for every $1\leq i\leq n$. This gives rise by functoriality to a morphism $\Gamma(\mathcal{P}re\mathcal{L}ie,\psi_{x,y_{1},\ldots,y_{n}}):\Gamma(\mathcal{P}re\mathcal{L}ie,E_{x,y_1,\ldots,y_n})\longrightarrow\Gamma(\mathcal{P}re\mathcal{L}ie,L)$. We set
    $$x\{y_{1},\ldots,y_{n}\}_{r_{1},\ldots,r_{n}}:=l(\Gamma(\mathcal{P}re\mathcal{L}ie,\psi_{x, y_{1},\ldots,y_{n}})(\mathcal{O}F_{\sum_i r_i}(e,\underbrace{e_{1},\ldots, e_{1}}_{r_{1}},\ldots,\underbrace{e_{n},\ldots, e_{n}}_{r_{n}}))),$$

    \noindent where $l$ is the $\Gamma(\mathcal{P}re\mathcal{L}ie,-)$-algebra structure on $L$. One can check that all the desired formulas are satisfied.\\

    Now suppose that $\mathbb{K}$ is an integral domain and that $L$ admits weighted brace operations $-\{-,\ldots,-\}_{r_1,\ldots,r_n}$. We first note that every elements in $\Gamma(\mathcal{P}re\mathcal{L}ie,L)$ can be described as a sum of iterated monadic compositions of corollas in some basis of homogeneous elements of $L$. This can be proved by using Proposition \ref{Oiso} and by following the same proofs of \cite[Theorem 5.1]{cesaro} and \cite[Lemma 5.2]{cesaro}, which come from the computation of the monadic composition in $\Gamma(\mathcal{P}re\mathcal{L}ie,-)$. We next pick elements $x, y_1,\ldots, y_n$ of the chosen basis such that $y_i\neq y_j$ if $i\neq j$, and set
    $$l(\mathcal{O}F_{\sum_i r_i}(x,\underbrace{y_1,\ldots,y_1}_{r_1},\ldots,\underbrace{y_n,\ldots,y_n}_{r_n})):=x\{y_1,\ldots,y_n\}_{r_1,\ldots,r_n}.$$

    \noindent Since every elements of $\Gamma(\mathcal{P}re\mathcal{L}ie,L)$ can be described as a composite of corollas in $\Gamma(\mathcal{P}re\mathcal{L}ie,-)$, we have defined $l:\Gamma(\mathcal{P}re\mathcal{L}ie,L)\longrightarrow L$. We see that this construction does not depend on the choice of the basis of the $L^k$'s. Indeed, we can apply the same arguments as in \cite[Lemma 5.15]{cesaro}, which only rely on computations and on the relations satisfied by weighted braces. The same proof of \cite[Lemma 5.18]{cesaro} can also be applied to prove that the resulting morphism $l$ endows $L$ with a structure of a $\Gamma(\mathcal{P}re\mathcal{L}ie,-)$-algebra.
\end{proof}

This theorem admits an analogue in the differential graded case.

\begin{thm}\label{prelie2}
    Every differential graded pre-Lie algebras with divided powers $L=\bigoplus_{k\in\mathbb{Z}} L^{k}$ admits weighted braces which satisfy the same formulas as in Theorem \ref{prelie1}, and which satisfy in addition the identity
$$d(x\{y_{1},\ldots,y_{n}\}_{r_{1},\ldots,r_{n}})=d(x)\{y_{1},\ldots,y_{n}\}_{r_{1},\ldots,r_{n}}+\sum_{k=1}^{n}(-1)^{\varepsilon_{k}} x\{y_{1},\ldots,y_{k},d(y_{k}),\ldots,y_{n}\}_{r_{1},\ldots,r_{k}-1,1,\ldots,r_{n}},$$

\noindent where $\varepsilon_{k}=|x|+|y_{1}|+\cdots+|y_{k-1}|$.\\

Conversely, if $\mathbb{K}$ is an integral domain, if $L^k$ is a free $\mathbb{K}$-module for every $k\in\mathbb{Z}$ and if $L$ admits weighted brace operations which satisfy the previous identity, then $L$ is a $\Gamma(\mathcal{P}re\mathcal{L}ie,-)$-algebra.
\end{thm}

\begin{proof}
    Let $x,y_1,\ldots,y_n\in L$. We set $E_{x,y_1,\ldots,y_n}$ to be the dg module generated by  elements $e,e_1,\ldots,e_n,f,f_1,\ldots,f_n$ such that $|e|=|x|,|e_1|=|y_1|,\ldots,|e_n|=|y_n|$ and $d(e)=f$, $d(e_1)=f_1,\ldots,d(e_n)=f_n$. We then have a morphism $\psi_{x,y_1,\ldots,y_n}:E_{x,y_1,\ldots,y_n}\longrightarrow L$ of dg modules defined by sending $e$ to $x$, and $e_i$ (resp. $f_i$) to $y_i$ (resp. $d(y_i)$) for every $1\leq i\leq n$. We set
    $$x\{y_{1},\ldots,y_{n}\}_{r_{1},\ldots,r_{n}}:=l(\Gamma(\mathcal{P}re\mathcal{L}ie,\psi_{x, y_{1},\ldots,y_{n}})(\mathcal{O}F_{\sum_i r_i}(e,\underbrace{e_{1},\ldots, e_{1}}_{r_{1}},\ldots,\underbrace{e_{n},\ldots, e_{n}}_{r_{n}}))),$$

    \noindent where $l$ is the $\Gamma(\mathcal{P}re\mathcal{L}ie,-)$-algebra structure on $L$. By forgetting the differentials and by applying Theorem \ref{prelie1}, we have that the operations $-\{-,...,-\}_{r_1,...,r_n}$ satisfy all the formulas of Theorem \ref{Cesaro}, with a sign. It only remains to prove the compatibility with the differential $d$.
    
    \begin{center}
        $\begin{array}{lll}
            d(x\{y_1,\ldots,y_n\}_{r_1,\ldots,r_n}) & = & dl(\Gamma(\mathcal{P}re\mathcal{L}ie,\psi_{x, y_{1},\ldots,y_{n}})(\mathcal{O}F_{\sum_i r_i}(e,\underbrace{e_1,\ldots,e_1}_{r_1}, \ldots,\underbrace{e_n,\ldots, e_n}_{r_n})))\\
            & = & l(\Gamma(\mathcal{P}re\mathcal{L}ie,\psi_{x, y_{1},\ldots,y_{n}})(d\mathcal{O}F_{\sum_i r_i}(e,\underbrace{e_1,\ldots, e_1}_{r_1}, \ldots,\underbrace{e_n,\ldots, e_n}_{r_n}))),
        \end{array}$
    \end{center}

    \noindent by commutation of $d$ with the algebra structure $l$ and $\psi_{x,y_1,...,y_n}$. Next, we claim that
    \begin{center}
        $\begin{array}{lll}
            d\mathcal{O}F_{\sum_i r_i}(e,\underbrace{e_1,\ldots, e_1}_{r_1}, \ldots,\underbrace{e_n,\ldots, e_n}_{r_n}) & = & \mathcal{O}F_{\sum_i r_i}(f,\underbrace{e_1,\ldots, e_1}_{r_1},\ldots,\underbrace{e_n,\ldots, e_n}_{r_n})\\
            & & \displaystyle+\sum_{k=1}^n\pm \mathcal{O}F_{\sum_i r_i}(e,\underbrace{e_1,\ldots, e_1}_{r_1},\ldots, \underbrace{e_k,\ldots, e_k}_{r_k-1}, f_k,\ldots,\underbrace{e_n, \ldots,e_n}_{r_n}).
        \end{array}$
    \end{center}

    \noindent Indeed, recall that 
    $$\mathcal{O}F_{\sum_i r_i}(e,\underbrace{e_1,\ldots, e_1}_{r_1}, \ldots,\underbrace{e_n,\ldots, e_n}_{r_n}) =\sum_{\sigma\in Sh(1,r_1,\ldots,r_n)}\sigma\cdot(F_{\sum_i r_i}\otimes e\otimes e_1^{\otimes r_1}\otimes\cdots\otimes e_n^{\otimes r_n}).$$

    \noindent We then have\\
    \quad\\
        $\displaystyle d\mathcal{O}F_{\sum_i r_i}(e,\underbrace{e_1,\ldots, e_1}_{r_1}, \ldots,\underbrace{e_n,\ldots, e_n}_{r_n})=\mathcal{O}F_{\sum_i r_i}(f,\underbrace{e_1,\ldots, e_1}_{r_1},\ldots,\underbrace{e_n,\ldots, e_n}_{r_n})$\\
        $$+\sum_{k=1}^n\sum_{\sigma\in Sh(1,r_1,\ldots,r_n)}\sum_{i=1}^{r_k}\pm\sigma\cdot(F_{\sum_i r_i}\otimes e\otimes e_1^{\otimes r_1}\otimes\cdots\otimes e_k^{\otimes i-1}\otimes f_k\otimes e_k^{r_k-i}\otimes\cdots\otimes e_n^{\otimes r_n}).$$

        \noindent Let $1\leq k\leq n$. For every $1\leq i\leq r_k$, we define $\tau_{k,i}$ as the permutation which permutes $f_k$ with the block $e_k^{\otimes r_k-i}$. We then obtain\\
        \quad\\
        $\displaystyle d\mathcal{O}F_{\sum_i r_i}(e,\underbrace{e_1,\ldots, e_1}_{r_1}, \ldots,\underbrace{e_n,\ldots, e_n}_{r_n})=\mathcal{O}F_{\sum_i r_i}(f,\underbrace{e_1,\ldots, e_1}_{r_1},\ldots,\underbrace{e_n,\ldots, e_n}_{r_n})$\\
        $$+\sum_{k=1}^n\sum_{\sigma\in Sh(1,r_1,\ldots,r_n)}\sum_{i=1}^{r_k}\pm\sigma\tau_{k,i}^{-1}\cdot(F_{\sum_i r_i}\otimes e\otimes e_1^{\otimes r_1}\otimes\cdots\otimes e_k^{\otimes r_k-1}\otimes f_k\otimes\cdots\otimes e_n^{\otimes r_n}).$$

        \noindent We note that, for every $1\leq i\leq r_k$, the permutation $\sigma\tau_{k,i}^{-1}$ is in $Sh(1,r_1,\ldots,r_k-1,1,\ldots,r_n)$. In the converse direction, if $\widetilde{\sigma}\in Sh(1,r_1,\ldots,r_k-1,1,\ldots,r_n)$, then there is a unique $1\leq i\leq r_k$ such that $\widetilde{\sigma}\tau_{k,i}\in Sh(1,r_1,\ldots,r_n)$. We thus have proved that\\
        \quad\\
        $\displaystyle d\mathcal{O}F_{\sum_i r_i}(e,\underbrace{e_1,\ldots, e_1}_{r_1}, \ldots,\underbrace{e_n,\ldots, e_n}_{r_n})=\mathcal{O}F_{\sum_i r_i}(f,\underbrace{e_1,\ldots, e_1}_{r_1},\ldots,\underbrace{e_n,\ldots, e_n}_{r_n})$\\
        $$+\sum_{k=1}^n\sum_{\widetilde{\sigma}\in Sh(1,r_1,\ldots,r_k-1,1,\ldots,r_n)}\pm\widetilde{\sigma}\cdot(F_{\sum_i r_i}\otimes e\otimes e_1^{\otimes r_1}\otimes\cdots\otimes e_k^{\otimes r_k-1}\otimes f_k\otimes\cdots\otimes e_n^{\otimes r_n})$$

        \noindent which gives
        \begin{center}
        $\begin{array}{lll}
            d\mathcal{O}F_{\sum_i r_i}(e,\underbrace{e_1,\ldots, e_1}_{r_1}, \ldots,\underbrace{e_n,\ldots, e_n}_{r_n}) & = & \mathcal{O}F_{\sum_i r_i}(f,\underbrace{e_1,\ldots, e_1}_{r_1},\ldots,\underbrace{e_n,\ldots, e_n}_{r_n})\\
            & & \displaystyle+\sum_{k=1}^n\pm \mathcal{O}F_{\sum_i r_i}(e,\underbrace{e_1,\ldots, e_1}_{r_1},\ldots, \underbrace{e_k,\ldots, e_k}_{r_k-1}, f_k,\ldots,\underbrace{e_n, \ldots,e_n}_{r_n}).
        \end{array}$
    \end{center}

    \noindent Applying $\Gamma(\mathcal{P}re\mathcal{L}ie,\psi_{x, y_{1},\ldots,y_{n}})$ and $l$ will then give the desired quantity, by definition of the weighted braces.\\

    Suppose now that $\mathbb{K}$ is an integral domain, that $L^k$ is free for every $k\in\mathbb{Z}$ and that $L$ is endowed with weighted brace operations. By Theorem \ref{prelie1}, and by forgetting the differential of $L$, we can define a morphism of graded modules $l:\Gamma(\mathcal{P}re\mathcal{L}ie,L)\longrightarrow L$ which is compatible with the monadic structure of $\Gamma(\mathcal{P}re\mathcal{L}ie,-)$. We now prove that $l$ commutes with the differential $d$. Since $d$ commutes with the monadic structure, it is sufficient to prove the commutation with $d$ when reducing to corollas. Let $x, y_1,\ldots, y_n$ be some basis elements with $y_i\neq y_j$ if $i\neq j$ and $r_1,\ldots, r_n\geq 1$. We then have
    \begin{center}
        $\begin{array}{lll}
            dl(\mathcal{O} F_{\sum_i r_i}(x,\underbrace{y_1,\ldots,y_1}_{r_1},\ldots,\underbrace{y_n,\ldots,y_n}_{r_n})) & = & d(x\{y_1,\ldots,y_n\}_{r_1,\ldots,r_n})\\
            & = & \displaystyle d(x)\{y_1,\ldots,y_n\}_{r_1,\ldots,r_n}\\
            & & \displaystyle+\sum_{k=1}^n\pm x\{y_1,\ldots,y_k,d(y_k),\ldots,y_n\}_{r_1,\ldots,r_k-1,1,\ldots,r_n}.\\
        \end{array}$
    \end{center}

    \noindent We decompose $d(y_k)$ in the chosen basis, and write
    $$d(y_k)=\sum_{\substack{i=1\\ i\neq k}}^n\lambda_{k,i} y_i+f_k,$$

    \noindent where $f_k=0$ or $f_k\notin Vect(y_1,\ldots,y_n)$ (note that $y_k$ cannot appear in the decomposition of $d(y_k)$ for degree reason). This gives
    \begin{center}
        $\begin{array}{lll}
            dl(\mathcal{O} F_{\sum_i r_i}(x,\underbrace{y_1,\ldots,y_1}_{r_1},\ldots,\underbrace{y_n,\ldots,y_n}_{r_n})) & = & d(x\{y_1,\ldots,y_n\}_{r_1,\ldots,r_n})\\
            & = & \displaystyle d(x)\{y_1,\ldots,y_n\}_{r_1,\ldots,r_n}\\
            & & \displaystyle+\sum_{k=1}^n\sum_{\substack{i=1\\ i\neq k}}^n\pm\lambda_{k,i}x\{y_1,\ldots,y_k,y_i,\ldots,y_n\}_{r_1,\ldots,r_k-1,1,\ldots,r_n}\\
            & & \displaystyle+\sum_{k=1}^n\pm x\{y_1,\ldots,y_k,f_k,\ldots,y_n\}_{r_1,\ldots,r_k-1,1,\ldots,r_n}
        \end{array}$
    \end{center}

    \noindent and then, by using the symmetry relations,
    \begin{center}
        $\begin{array}{lll}
            dl(\mathcal{O} F_{\sum_i r_i}(x,\underbrace{y_1,\ldots,y_1}_{r_1},\ldots,\underbrace{y_n,\ldots,y_n}_{r_n})) & = & d(x\{y_1,\ldots,y_n\}_{r_1,\ldots,r_n})\\
            & = & \displaystyle d(x)\{y_1,\ldots,y_n\}_{r_1,\ldots,r_n}\\
            & & \displaystyle+\sum_{k=1}^n\sum_{\substack{i=1\\ i\neq k}}^n\pm\lambda_{k,i}x\{y_1,\ldots,y_i,y_i,\ldots,y_k,\ldots,y_n\}_{r_1,\ldots,r_i,1,\ldots,r_k-1,\ldots,r_n}\\
            & & \displaystyle+\sum_{k=1}^n\pm x\{y_1,\ldots,y_k,f_k,\ldots,y_n\}_{r_1,\ldots,r_k-1,1,\ldots,r_n}
        \end{array}$
    \end{center}

    \noindent The second sum can be simplified, depending on the parity of $|y_i|$ for every $i$. If $|y_i|$ is even, then by formula $(iv)$ of Theorem \ref{Cesaro}, we have
    $$x\{y_1,\ldots,y_i,y_i,\ldots,y_k,\ldots,y_n\}_{r_1,\ldots,r_i,1,\ldots,r_k-1,\ldots,r_n}=(r_i+1)x\{y_1,\ldots,y_i,\ldots,y_k,\ldots,y_n\}_{r_1,\ldots,r_i+1,\ldots,r_k-1,\ldots,r_n}.$$

    \noindent If $|y_i|$ is odd, then since we have supposed that any odd degree element has an associated weight equal to $1$, we have that $r_i=1$. By using the symmetry relation, and using that we have no $2$-torsion elements since $L$ is free and that $\mathbb{K}$ is an integral domain, we have that
    $$x\{y_1,\ldots,y_i,y_i,\ldots,y_k,\ldots,y_n\}_{r_1,\ldots,1,1,\ldots,r_k-1,\ldots,r_n}=0.$$

    \noindent We finally have that
    \begin{center}
        $\begin{array}{lll}
            dl(\mathcal{O} F_{\sum_i r_i}(x,\underbrace{y_1,\ldots,y_1}_{r_1},\ldots,\underbrace{y_n,\ldots,y_n}_{r_n})) & = & d(x\{y_1,\ldots,y_n\}_{r_1,\ldots,r_n})\\
            & = & \displaystyle d(x)\{y_1,\ldots,y_n\}_{r_1,\ldots,r_n}\\
            & + & \displaystyle\sum_{k=1}^n\sum_{\substack{i=1\\ i\neq k}}^n\pm\delta_i\lambda_{k,i}(r_i+1)x\{y_1,\ldots,y_i,\ldots,y_k,\ldots,y_n\}_{r_1,\ldots,r_i+1,\ldots,r_k-1,\ldots,r_n}\\
            & + & \displaystyle\sum_{k=1}^n\pm x\{y_1,\ldots,y_k,f_k,\ldots,y_n\}_{r_1,\ldots,r_k-1,1,\ldots,r_n}
        \end{array}$
    \end{center}

    \noindent where $\delta_i=0$ if $|y_i|$ is odd, and $\delta_i=1$ else.\\

    We now compute $ld(\mathcal{O} F_{\sum_i r_i}(x,\underbrace{y_1,\ldots,y_1}_{r_1},\ldots,\underbrace{y_n,\ldots,y_n}_{r_n}))$. By the same computations as the beginning of the proof, we have that\\
    \quad\\
    $\displaystyle d\mathcal{O}F_{\sum_i r_i}(x,\underbrace{y_1,\ldots, y_1}_{r_1}, \ldots,\underbrace{y_n,\ldots, y_n}_{r_n})=\mathcal{O}F_{\sum_i r_i}(d(x),\underbrace{y_1,\ldots, y_1}_{r_1},\ldots,\underbrace{y_n,\ldots, y_n}_{r_n})$\\
        $$+\sum_{k=1}^n\sum_{\widetilde{\sigma}\in Sh(1,r_1,\ldots,r_k-1,1,\ldots,r_n)}\pm\widetilde{\sigma}\cdot(F_{\sum_i r_i}\otimes x\otimes y_1^{\otimes r_1}\otimes\cdots\otimes y_k^{\otimes r_k-1}\otimes d(y_k)\otimes\cdots\otimes y_n^{\otimes r_n}).$$

    \noindent We now fix $1\leq k\leq n$. We aim to compute the sum
$$\sum_{\widetilde{\sigma}\in Sh(1,r_1,\ldots,r_k-1,1,\ldots,r_n)}\pm\widetilde{\sigma}\cdot(F_{\sum_i r_i}\otimes x\otimes y_1^{\otimes r_1}\otimes\cdots\otimes y_k^{\otimes r_k-1}\otimes d(y_k)\otimes\cdots\otimes y_n^{\otimes r_n}).$$
    
    \noindent We use the decomposition of $d(y_k)$ in the chosen basis. The $f_k$ part will precisely give
    $$\mathcal{O}F_{\sum_i r_i}(x,\underbrace{y_1,\ldots,y_1}_{r_1},\ldots,\underbrace{y_k,\ldots,y_k}_{r_k-1},f_k,\ldots,\underbrace{y_n,\ldots,y_n}_{r_n}).$$

    \noindent We now look at the other terms. These terms give the sum
    $$\sum_{\substack{i=1\\ i\neq k}}^n\sum_{\widetilde{\sigma}\in Sh(1,r_1,\ldots,r_k-1,1,\ldots,r_n)}\pm\lambda_{k,i}(\widetilde{\sigma}\cdot(F_{\sum_i r_i}\otimes x\otimes y_1^{\otimes r_1}\otimes \cdots\otimes y_k^{\otimes r_k-1}\otimes y_i\otimes\cdots\otimes y_n^{\otimes r_n})).$$

    \noindent By putting the only $y_i$ with the others, we find that this sum is equal to
    $$\sum_{\substack{i=1\\ i\neq k}}^n\sum_{\widetilde{\sigma}\in Sh(1,r_1,\ldots,r_i,1,\ldots,r_k-1,\ldots,r_n)}\pm\lambda_{k,i}(\widetilde{\sigma}\cdot(F_{\sum_i r_i}\otimes x\otimes y_1^{\otimes r_1}\otimes\cdots\otimes y_i^{\otimes r_i}\otimes y_i\otimes\cdots\otimes y_k^{\otimes r_k-1}\otimes\cdots\otimes y_n^{\otimes r_n})).$$

    \noindent We now use that for every $\widetilde{\sigma}\in Sh(1,r_1,\ldots,r_i,1,\ldots,r_{k}-1,\ldots,r_n)$, there exists a unique permutation permutation $\tau_{i,j}$, defined by inserting the last occurrence of $y_i$ among $\underbrace{y_i\otimes\cdots\otimes y_i}_{r_i}$ in position $j$ (so that $1\leq j\leq r_i+1$), such that $\widetilde{\sigma}\tau_{i,j}^{-1}\in Sh(1,r_1,\ldots,r_i+1,\ldots,r_k-1,\ldots,r_n)$. In the other direction, for every $\sigma\in Sh(1,r_1,\ldots,r_i+1,\ldots,r_k-1,\ldots,r_n)$ and $1\leq j\leq r_i+1$, we have that $\sigma\tau_{i,j}\in Sh(1,r_1,\ldots,r_i,1,\ldots,r_k-1,\ldots,r_n)$. We thus obtain the sum
    $$\sum_{\substack{i=1\\ i\neq k}}^n\sum_{{\sigma}\in Sh(1,r_1,\ldots,r_i+1,\ldots,r_k-1,\ldots,r_n)}\sum_{j=1}^{r_i+1}\pm\lambda_{k,i}({\sigma\tau_{i,j}}\cdot(F_{\sum_i r_i}\otimes x\otimes y_1^{\otimes r_1}\otimes\cdots\otimes y_i^{\otimes r_i+1}\otimes\cdots\otimes y_k^{\otimes r_k-1}\otimes\cdots\otimes y_n^{\otimes r_n})).$$
    
    \noindent For a fixed $i\neq k$, we need to distinguish two cases: either $|y_i|$ is even, or $|y_i|$ is odd. In the first case, we have that this sum is
    $$\sum_{\substack{i=1\\ i\neq k}}^n\sum_{{\sigma}\in Sh(1,r_1,\ldots,r_i+1,\ldots,r_k-1,\ldots,r_n)}\sum_{j=1}^{r_i+1}\pm\lambda_{k,i}({\sigma}\cdot(F_{\sum_i r_i}\otimes x\otimes y_1^{\otimes r_1}\otimes\cdots\otimes y_i^{\otimes r_i+1}\otimes\cdots\otimes y_k^{\otimes r_k-1}\otimes\cdots\otimes y_n^{\otimes r_n}))$$

    \noindent which is precisely
$$\sum_{\substack{i=1\\ i\neq k}}^n\pm\lambda_{k,i} (r_i+1)\mathcal{O} F_{\sum_i r_i}(x,\underbrace{y_1,\ldots,y_1}_{r_1},\ldots,\underbrace{y_i,\ldots,y_i}_{r_i+1},\ldots,\underbrace{y_k,\ldots,y_k}_{r_k-1},\ldots,\underbrace{y_n,\ldots,y_n}_{r_n}).$$

\noindent In the second case, since we have supposed that every odd degree element has an associated weight equal to $1$, we have that $r_i=1$. We have that $\tau_{i,1}$ permutes the two $y_i$'s, which gives a sign, while $\tau_{i,2}=id$. The sum of the two obtained elements is then $0$.\\
    
    \noindent We thus have proved that
    \quad\\

    \noindent $d\mathcal{O} F_{\sum_i r_i}(x,\underbrace{y_1,\ldots,y_1}_{r_1},\ldots,\underbrace{y_n,\ldots,y_n}_{r_n}) = \mathcal{O} F_{\sum_i r_i}(d(x),\underbrace{y_1,\ldots,y_1}_{r_1},\ldots,\underbrace{y_n,\ldots,y_n}_{r_n})$
    \begin{flushright}
        $\displaystyle + \sum_{k=1}^n\sum_{\substack{i=1\\ i\neq k}}^n\pm\delta_i\lambda_{k,i} (r_i+1)\mathcal{O} F_{\sum_i r_i}(x,\underbrace{y_1,\ldots,y_1}_{r_1},\ldots,\underbrace{y_i,\ldots,y_i}_{r_i+1},\ldots,\underbrace{y_k,\ldots,y_k}_{r_k-1},\ldots,\underbrace{y_n,\ldots,y_n}_{r_n})$\\
    \end{flushright}
    \begin{center}$\displaystyle +\sum_{k=1}^n\pm\mathcal{O} F_{\sum_i r_i}(x,\underbrace{y_1,\ldots,y_1}_{r_1},\ldots,\underbrace{y_k,\ldots,y_k}_{r_k-1}, f_k,\ldots,\underbrace{y_n,\ldots,y_n}_{r_n}).$\end{center}

    \noindent where we have set $\delta_i=0$ if $|y_i|$ is odd, and $\delta_i=1$ else. By definition of $l$, we have
    \begin{center}
        $\begin{array}{lll}
            l(d\mathcal{O} F_{\sum_i r_i}(x,\underbrace{y_1,\ldots,y_1}_{r_1},\ldots,\underbrace{y_n,\ldots,y_n}_{r_n})) & = &d(x)\{y_1,\ldots,y_n\}_{r_1,\ldots,r_n} \\
           & + &\displaystyle \sum_{k=1}^n\sum_{\substack{i=1\\ i\neq k}}^n\pm\delta_i\lambda_{k,i} (r_i+1)x\{y_1,\ldots,y_i,\ldots,y_k,\ldots,y_n\}_{r_1,\ldots,r_{i}+1,\ldots,r_i-1,\ldots,r_n}\\
           & + & \displaystyle \sum_{k=1}^n\pm x\{y_1,\ldots,y_k, f_k,\ldots,y_n\}_{r_1,\ldots,r_k-1,1,\ldots,r_n}.
        \end{array}$
    \end{center}

    \noindent which proves that $l$ commutes with $d$.
\end{proof}

We then deduce from Propositions \ref{prelie1} and \ref{prelie2} that every differential graded pre-Lie algebra with divided powers is in particular a differential graded pre-Lie algebra, with
$$x\star y=x\{y\}_{1}.$$

\begin{remarque}\label{rembraces}
If $\mathbb{Q}\subset\mathbb{K}$ and if $L$ is a differential graded pre-Lie algebra, then $L$ is a differential graded pre-Lie algebra with divided powers whose weighted braces are explicitly given by
$$x\{y_{1},\ldots,y_{n}\}_{r_{1},\ldots,r_{n}}=\frac{1}{\prod_{i}r_{i}!}x\{\underbrace{y_{1},\ldots,y_{1}}_{r_{1}},\ldots,\underbrace{y_{n},\ldots,y_{n}}_{r_{n}}\}$$

\noindent in terms of symmetric braces.
\end{remarque}

\begin{remarque}
    Every morphism of $\Gamma(\mathcal{P}re\mathcal{L}ie,-)$-algebras preserves the weighted braces:
    $$f(x\{y_{1},\ldots,y_{n}\}_{r_{1},\ldots,r_{n}})=f(x)\{f(y_{1}),\ldots,f(y_{n})\}_{r_{1},\ldots,r_{n}}.$$
\end{remarque}

In order to perform infinite sums, we define the notion of a \textit{complete} $\Gamma(\mathcal{P}re\mathcal{L}ie,-)$-algebra. We recall the following definition.

\begin{defi}
    A {\normalfont filtered} dg module is the data of a dg module $L$ with inclusions of dg modules
    $$\cdots\subset F_n L\subset F_{n-1}L\subset\cdots\subset F_1L =L.$$

    A filtered dg module is {\normalfont complete} if the morphism $L\longrightarrow \lim_{n\geq 1} L/F_nL$ is an isomorphism.
\end{defi}

In general, for every filtered dg module $L$, the dg module $\widehat{L}=\lim_{n\geq 1}L/F_nL$ is a complete dg module with the filtration $F_n\widehat{L}=Ker(\widehat{L}\longrightarrow L/F_nL)$, since $\widehat{L}/F_n\widehat{L}\simeq L/F_nL$.\\

We now define the notion of a filtered $\Gamma(\mathcal{P}re\mathcal{L}ie,-)$-algebra.

\begin{defi}\label{filteredgamma}
    A {\normalfont filtered} $\Gamma(\mathcal{P}re\mathcal{L}ie,-)$-algebra is a $\Gamma(\mathcal{P}re\mathcal{L}ie,-)$-algebra endowed with a filtration preserved by the weighted braces in the sense that
    $$F_{k}L\{F_{k_{1}}L,\ldots,F_{k_{n}}L\}_{r_{1},\ldots,r_{n}}\subset F_{k+k_{1}r_{1}+\cdots+k_{n}r_{n}}L.$$

    A filtered $\Gamma(\mathcal{P}re\mathcal{L}ie,-)$-algebra is {\normalfont complete} if $L$ is complete as a filtered dg module.
\end{defi}

If $L$ is a filtered $\Gamma(\mathcal{P}re\mathcal{L}ie,-)$-algebra, then the weighted braces $-\{-,\ldots,-\}_{r_1,\ldots,r_n}$ induce weighted braces on the completion $\widehat{L}=\lim_{n\geq 1}L/F_nL$, which satisfy the formulas of Theorems \ref{prelie1} and \ref{prelie2}, and preserve the filtration on $\widehat{L}$, so that $\widehat{L}$ forms a complete $\Gamma(\mathcal{P}re\mathcal{L}ie,-)$-algebra (provided that we work over a field).

\subsubsection{Examples of $\Gamma(\mathcal{P}re\mathcal{L}ie,-)$-algebras}

We give examples of dg pre-Lie algebras with divided powers. The first examples are given by dg brace algebras, following the idea of the proof in the non graded framework in \cite{cesaro}.

\begin{defi}
    A {\normalfont differential graded brace algebra} is a differential graded module $L$ endowed with brace operations
    $$-\langle-,\ldots,-\rangle:L^{\otimes n+1}\longrightarrow L$$

    \noindent which are compatible with the differential $d$:
    $$d(f\langle g_{1},\ldots,g_{n}\rangle)=d(f)\langle g_{1},\ldots,g_{n}\rangle+\sum_{k=1}^{n}\pm f\langle g_{1},\ldots,d(g_{k}),\ldots,g_{n}\rangle,$$

    \noindent and such that $f\langle\rangle=f$ and
    $$f\langle g_1,\ldots,g_n\rangle\langle h_1,\ldots,h_r\rangle=\sum \pm f\langle H_1,g_1\langle H_2\rangle,\ldots,H_{2n-1}, g_n\langle H_{2n}\rangle,H_{2n+1}\rangle,$$

    \noindent where the sum is over all consecutive subsets $H_1\sqcup H_2\sqcup \cdots\sqcup H_{2n+1}=\{h_1,\ldots,h_r\}$, and the sign is yielded by the permutation of the $g_i$'s with the $h_j$'s.
\end{defi}

The operad which governs brace algebras is denoted by $\mathcal{B}race$, and is defined, in arity $n$, as the $\mathbb{K}$-module spanned by the planar n-trees, i.e. trees with an order on the set of inputs for each vertex (see \cite[$\mathsection$6.1]{cesaro} or \cite[$\mathsection$2]{chapotonbis} for some details on the operad $\mathcal{B}race$).\\

This operad allows us to represent all operations in brace algebras by the action of a planar tree, or by a planar tree labeled with the inputs. For instance, we have

\begin{equation*}
    \begin{tikzpicture}[baseline={([yshift=-.5ex]current bounding box.center)},scale=0.6]
    \node[draw,circle,scale=0.6] (i) at (0,0) {$f$};
    \node[draw,circle,scale=0.6] (1) at (-1,1) {$g_1$};
    \node[draw,circle,scale=0.6] (1b) at (-1.5,2) {$h_1$};
    \node[draw,circle,scale=0.6] (1bb) at (-0.5,2) {$h_2$};
    \node[draw,circle,scale=0.6] (2) at (0,1) {$g_2$};
    \node[draw,circle,scale=0.6] (3) at (1,1) {$g_3$};
    \node[draw,circle,scale=0.6] (3b) at (1,2) {$h_{3}$};
    \draw (2) -- (i);
    \draw (1) -- (1b);
    \draw (1) -- (1bb);
    \draw (i) -- (1);
    \draw (i) -- (3);
    \draw (3) -- (3b);
    \end{tikzpicture}=f\langle g_{1}\langle h_{1},h_{2}\rangle,g_{2},g_{3}\langle h_{3}\rangle\rangle.
\end{equation*}

\begin{remarque}
    Because the action of the symmetric groups on $\mathcal{B}race$ is free, we have that the trace map induces an isomorphism of monads $Tr:\mathcal{S}(\mathcal{B}race,-)\longrightarrow\Gamma(\mathcal{B}race,-)$.
\end{remarque}

We have an inclusion
$$i:\mathcal{P}re\mathcal{L}ie\hookrightarrow \mathcal{B}race$$

\noindent defined by the \textit{symmetrization} of trees. Namely, $i$ is obtained by summing over all possible ways to write a given tree $t$ as a planar tree. For instance:
\begin{equation*}
    i\left(
    \begin{tikzpicture}[baseline={([yshift=-.5ex]current bounding box.center)},scale=0.6]
    \node[draw,circle,scale=0.6] (i) at (0,0) {$1$};
    \node[draw,circle,scale=0.6] (1) at (-1,1) {$2$};
    \node[draw,circle,scale=0.6] (2) at (0,1) {$3$};
    \node[draw,circle,scale=0.6] (3) at (1,1) {$4$};
    \draw (2) -- (i);
    \draw (i) -- (1);
    \draw (i) -- (3);
    \end{tikzpicture}\right)=\begin{tikzpicture}[baseline={([yshift=-.5ex]current bounding box.center)},scale=0.6]
    \node[draw,circle,scale=0.6] (i) at (0,0) {$1$};
    \node[draw,circle,scale=0.6] (1) at (-1,1) {$2$};
    \node[draw,circle,scale=0.6] (2) at (0,1) {$3$};
    \node[draw,circle,scale=0.6] (3) at (1,1) {$4$};
    \draw (2) -- (i);
    \draw (i) -- (1);
    \draw (i) -- (3);
    \end{tikzpicture}+\begin{tikzpicture}[baseline={([yshift=-.5ex]current bounding box.center)},scale=0.6]
    \node[draw,circle,scale=0.6] (i) at (0,0) {$1$};
    \node[draw,circle,scale=0.6] (1) at (-1,1) {$3$};
    \node[draw,circle,scale=0.6] (2) at (0,1) {$2$};
    \node[draw,circle,scale=0.6] (3) at (1,1) {$4$};
    \draw (2) -- (i);
    \draw (i) -- (1);
    \draw (i) -- (3);
    \end{tikzpicture}+\begin{tikzpicture}[baseline={([yshift=-.5ex]current bounding box.center)},scale=0.6]
    \node[draw,circle,scale=0.6] (i) at (0,0) {$1$};
    \node[draw,circle,scale=0.6] (1) at (-1,1) {$4$};
    \node[draw,circle,scale=0.6] (2) at (0,1) {$3$};
    \node[draw,circle,scale=0.6] (3) at (1,1) {$2$};
    \draw (2) -- (i);
    \draw (i) -- (1);
    \draw (i) -- (3);
    \end{tikzpicture}+\begin{tikzpicture}[baseline={([yshift=-.5ex]current bounding box.center)},scale=0.6]
    \node[draw,circle,scale=0.6] (i) at (0,0) {$1$};
    \node[draw,circle,scale=0.6] (1) at (-1,1) {$2$};
    \node[draw,circle,scale=0.6] (2) at (0,1) {$4$};
    \node[draw,circle,scale=0.6] (3) at (1,1) {$3$};
    \draw (2) -- (i);
    \draw (i) -- (1);
    \draw (i) -- (3);
    \end{tikzpicture}+\begin{tikzpicture}[baseline={([yshift=-.5ex]current bounding box.center)},scale=0.6]
    \node[draw,circle,scale=0.6] (i) at (0,0) {$1$};
    \node[draw,circle,scale=0.6] (1) at (-1,1) {$4$};
    \node[draw,circle,scale=0.6] (2) at (0,1) {$2$};
    \node[draw,circle,scale=0.6] (3) at (1,1) {$3$};
    \draw (2) -- (i);
    \draw (i) -- (1);
    \draw (i) -- (3);
    \end{tikzpicture}+\begin{tikzpicture}[baseline={([yshift=-.5ex]current bounding box.center)},scale=0.6]
    \node[draw,circle,scale=0.6] (i) at (0,0) {$1$};
    \node[draw,circle,scale=0.6] (1) at (-1,1) {$3$};
    \node[draw,circle,scale=0.6] (2) at (0,1) {$4$};
    \node[draw,circle,scale=0.6] (3) at (1,1) {$2$};
    \draw (2) -- (i);
    \draw (i) -- (1);
    \draw (i) -- (3);
    \end{tikzpicture}\ \ .
\end{equation*}

The map $i$ induces a morphism of monads that can be used to define a $\Gamma(\mathcal{P}re\mathcal{L}ie,-)$-algebra structure on every dg brace algebra $L$, given by the following composition:

\[\begin{tikzcd}
	{\Gamma(\mathcal{P}re\mathcal{L}ie,L)} & {\Gamma(\mathcal{B}race,L)} & {\mathcal{S}(\mathcal{B}race,L)} & L,
	\arrow["{\Gamma(i,L)}", from=1-1, to=1-2]
	\arrow["\simeq", from=1-3, to=1-2]
	\arrow["Tr"', from=1-3, to=1-2]
	\arrow["l", from=1-3, to=1-4]
\end{tikzcd}\]

\noindent where we denote by $l:\mathcal{S}(\mathcal{P}re\mathcal{L}ie,L)\longrightarrow L$ the $\mathcal{B}race$-algebra structure. We aim to compute the weighted braces.

\begin{thm}\label{braces}
    Every dg brace algebra $L$ is endowed with a $\Gamma(\mathcal{P}re\mathcal{L}ie,-)$-algebra structure. Moreover, weighted braces $-\{-,\ldots,-\}_{r_{1},\ldots,r_{n}}$ are explicitly given by
    $$f\{g_{1},\ldots,g_{n}\}_{r_{1},\ldots,r_{n}}=\sum_{\sigma\in Sh(r_{1},\ldots,r_{n})}\pm f\langle \overline{g}_{\sigma^{-1}(1)},\ldots,\overline{g}_{\sigma^{-1}(r)}\rangle,$$

    \noindent where we have set $r=\sum_{i}r_{i}$ and $(\overline{g}_{1},\ldots,\overline{g}_{r})=(\underbrace{g_{1},\ldots,g_{1}}_{r_{1}},\ldots,\underbrace{g_{n},\ldots,g_{n}}_{r_{n}})$.
\end{thm}

\begin{proof}
    Let $L$ be a brace algebra. As seen before, we have a $\Gamma(\mathcal{P}re\mathcal{L}ie,-)$-algebra structure on $L$ given by the composite
\[\begin{tikzcd}
	{\Gamma(\mathcal{P}re\mathcal{L}ie,L)} & {\Gamma(\mathcal{B}race,L)} & {\mathcal{S}(\mathcal{B}race,L)} & L
	\arrow["{\Gamma(i,L)}", from=1-1, to=1-2]
	\arrow["\simeq", from=1-3, to=1-2]
	\arrow["Tr"', from=1-3, to=1-2]
	\arrow["l", from=1-3, to=1-4]
\end{tikzcd}\]

    \noindent where $l:\mathcal{S}(\mathcal{B}race,L)\longrightarrow L$ is the brace algebra structure. We now compute the weighted braces. Let $f,g_1,\ldots,g_n\in L$ be homogeneous elements with $g_i\neq g_j$ whenever $i\neq j$ and $r_1,\ldots,r_n\geq 0$ (recall that we have suppose that, in the situation $char(\mathbb{K})\neq 2$, any odd degree element has an associated weight equaled to $0$ or $1$). We set $E=E_{f,g_1,\ldots,g_n}$ and $\psi=\psi_{f,g_1,\ldots,g_n}$ (see the proof of Theorem \ref{prelie2}). We use the following commutative diagram
\[\begin{tikzcd}
	{\Gamma(\mathcal{P}re\mathcal{L}ie,L)} & {\Gamma(\mathcal{B}race,L)} & {\mathcal{S}(\mathcal{B}race,L)} & L \\
	{\Gamma(\mathcal{P}re\mathcal{L}ie,E)} & {\Gamma(\mathcal{B}race,E)} & {\mathcal{S}(\mathcal{B}race,E).}
	\arrow["{\Gamma(i,L)}", from=1-1, to=1-2]
	\arrow["\simeq", from=1-3, to=1-2]
	\arrow["Tr"', from=1-3, to=1-2]
	\arrow["l", from=1-3, to=1-4]
	\arrow["{\Gamma(\mathcal{P}re\mathcal{L}ie,\psi)}", from=2-1, to=1-1]
	\arrow["{\Gamma(i,E)}"', from=2-1, to=2-2]
	\arrow["{\Gamma(\mathcal{B}race,\psi)}", from=2-2, to=1-2]
	\arrow["{\mathcal{S}(\mathcal{B}race,\psi)}", from=2-3, to=1-3]
	\arrow["\simeq"', from=2-3, to=2-2]
	\arrow["Tr", from=2-3, to=2-2]
\end{tikzcd}\]
    
    \noindent We keep the notations $f,g_1,\ldots,g_n$ for the corresponding elements in $E$. Then the element $f\{g_1,\ldots,g_n\}_{r_1,\ldots,r_n}$ is given by the image of $x=\mathcal{O}F_{r}(f,\underbrace{g_{1},\ldots,g_{1}}_{r_{1}},\ldots,\underbrace{g_{n},\ldots,g_{n}}_{r_{n}})\in\Gamma(\mathcal{P}re\mathcal{L}ie,E)$ under the composite
\[\begin{tikzcd}
	{\Gamma(\mathcal{P}re\mathcal{L}ie,E)} && {\Gamma(\mathcal{P}re\mathcal{L}ie,L)} & {\Gamma(\mathcal{B}race,L)} & {\mathcal{S}(\mathcal{B}race,L)} & L.
	\arrow["{\Gamma(\mathcal{P}re\mathcal{L}ie,\psi)}", from=1-1, to=1-3]
	\arrow["{\Gamma(i,L)}", from=1-3, to=1-4]
	\arrow["\simeq", from=1-5, to=1-4]
	\arrow["Tr"', from=1-5, to=1-4]
	\arrow["l", from=1-5, to=1-6]
\end{tikzcd}\]
    
    \noindent Our goal is to compute the image of $x$ under the bottom composite of the diagram, which is
\[\begin{tikzcd}
	{\Gamma(\mathcal{P}re\mathcal{L}ie,E)} & {\Gamma(\mathcal{B}race,E)} & {\mathcal{S}(\mathcal{B}race,E)} && {\mathcal{S}(\mathcal{B}race,L).}
	\arrow["{\Gamma(i,E)}", from=1-1, to=1-2]
	\arrow["\simeq", from=1-3, to=1-2]
	\arrow["Tr"', from=1-3, to=1-2]
	\arrow["{\mathcal{S}(\mathcal{B}race,\psi)}", from=1-3, to=1-5]
\end{tikzcd}\]
    
    \noindent We set $(\overline{g}_{1},\ldots,\overline{g}_{r+1})=(f,\underbrace{g_{1},\ldots,g_{1}}_{r_{1}},\ldots,\underbrace{g_{n},\ldots,g_{n}}_{r_{n}})$ (note that we have added $f$ here so that these $\overline{g}_{i}$'s are different from the $\overline{g}_{i}$'s of the theorem). We then precisely have:
    

    $$x=\sum_{\sigma\in\Sigma_{r+1}/\prod_{i}\Sigma_{r_{i}}}\pm(\sigma\cdot F_{r})\otimes\overline{g}_{\sigma^{-1}(1)}\otimes\cdots\otimes\overline{g}_{\sigma^{-1}(r+1)}$$

    \noindent by definition of the orbit map (see Lemma \ref{technique2}). Now, because $\Sigma_{r+1}/\prod_{i}\Sigma_{r_{i}}$ is in bijection with $Sh(1,r_{1},\ldots,r_{n})$, we can write $x$ as
    $$x=\sum_{\sigma\in Sh(1,r_{1},\ldots,r_{n})}\pm(\sigma\cdot F_{r})\otimes \overline{g}_{\sigma^{-1}(1)}\otimes\cdots\otimes\overline{g}_{\sigma^{-1}(r+1)}.$$

    \noindent We now embed $\mathcal{P}re\mathcal{L}ie$ into $\mathcal{B}race$. The tree $F_{r}$ can be seen in $\mathcal{B}race$ as $\sum_{\substack{s\in\Sigma_{r+1}\\ s(1)=1}}s\cdot\overline{F_{r}}$ where $\overline{F_{r}}$ is the planar tree
    \begin{equation*}
    \overline{F_{r}}=
    \begin{tikzpicture}[baseline={([yshift=-.5ex]current bounding box.center)},scale=0.6]
    \node[draw,circle,scale=0.6] (i) at (0,0) {$1$};
    \node[draw,circle,scale=0.6] (1) at (-1,1) {$2$};
    \node[draw,circle,scale=0.6] (2) at (0,1) {$3$};
    \node[draw,circle,scale=0.6] (3) at (2,1) {$r+1$};
    \node (a) at (0.9,1) {$\cdots$};
    \draw (2) -- (i);
    \draw (i) -- (1);
    \draw (i) -- (3);
    \end{tikzpicture}\ \
\end{equation*}

\noindent We then obtain, in $\Gamma(\mathcal{B}race,E)$,
$$x=\sum_{\sigma\in Sh(1,r_{1},\ldots,r_{n})}\sum_{\substack{s\in\Sigma_{r+1}\\ s(1)=1}}\pm(\sigma s\cdot\overline{F_{r}})\otimes \overline{g}_{\sigma^{-1}(1)}\otimes\cdots\otimes\overline{g}_{\sigma^{-1}(r+1)}.$$

\noindent We now need to compute $y=Tr^{-1}(x)\in\mathcal{S}(\mathcal{B}race,E)$. We claim that
$$y=\sum_{\substack{\omega\in Sh(1,r_{1},\ldots,r_{n})\\\omega(1)=1}}\pm\overline{F_{r}}(\overline{g}_{\omega^{-1}(1)},\ldots,\overline{g}_{\omega^{-1}(r+1)}).$$

\noindent We compute
$$Tr(y)=\sum_{\substack{\omega\in Sh(1,r_{1},\ldots,r_{n})\\\omega(1)=1}}\sum_{\tau\in\Sigma_{r+1}}\pm(\tau\cdot\overline{F_{r}})\otimes \overline{g}_{\omega^{-1}\tau^{-1}(1)}\otimes\cdots\otimes\overline{g}_{\omega^{-1}\tau^{-1}(r+1)}.$$

\noindent The fact that $Tr(y)=x$ comes from the existence of a bijective correspondance 
$$\varphi:Sh(1,r_1,\ldots,r_n)\times\{s\in\Sigma_{r+1}\ |\ s(1)=1\}\longrightarrow\{\omega\in Sh(1,r_1,\ldots,r_n)\ |\ \omega(1)=1\}\times\Sigma_{r+1}.$$

\noindent Indeed, let $\sigma\in Sh(1,r_1,\ldots,r_n)$ and $s\in\Sigma_{r+1}$ be such that $s(1)=1$. We set $\tau=\sigma s$, and decompose $\tau^{-1}\sigma=s^{-1}=\omega\mu$ as a product of $\omega\in Sh(1,r_1,\ldots,r_n)$ with $\mu\in\Sigma_1\times\Sigma_{r_1}\times\cdots\times\Sigma_{r_n}$. Since $s(1)=1$, we have that $\omega(1)=1$. We then set $\varphi(\sigma,s):=(\omega,\tau)$.  Since the couple $(\omega,\mu)$ uniquely depends on $s$, we have a well defined injective map $\varphi$ between two sets with the same cardinal. The map $\varphi$ is then a bijection.\\

\noindent Since we have $\overline{g}_{\mu\sigma^{-1}(i)}=\overline{g}_{\sigma^{-1}(i)}$ for every $i$, we obtain that $Tr(y)=x$, which proves the theorem.
\end{proof}

\begin{cor}\label{preliebrace}
    Let $\mathcal{P}$ be a non symmetric dg operad with $\mathcal{P}(0)=0$. We denote by $1\in\mathcal{P}(1)$ the unit element of $\mathcal{P}$, by $p(q_1,\ldots,q_n)=p\otimes q_1\otimes\cdots\otimes q_n\in\mathcal{P}\circ\mathcal{P}$ and by $\gamma:\mathcal{P}\circ\mathcal{P}\longrightarrow\mathcal{P}$ the operadic composition of $\mathcal{P}$. Then the dg module $\bigoplus_{r\geq 1}\mathcal{P}(r)$ admits a structure of a $\Gamma(\mathcal{P}re\mathcal{L}ie,-)$-algebra induced by the following brace algebra structure:
    $$p\langle q_1,\ldots,q_n\rangle=\sum_{1\leq i_1<\cdots<i_n\leq r}  \gamma(p(1,\ldots,\underset{i_1}{q_1},\ldots,\underset{i_n}{q_n},\ldots,1)).$$

    \noindent We also set $p\langle q_1,\ldots,q_n\rangle=0$ if the operadic composition is not possible.
\end{cor}

\begin{proof}
    We refer to \cite{gerstenhaber2} for the brace algebra structure of $\bigoplus_{r\geq 1}\mathcal{P}(r)$. It is endowed with a $\Gamma(\mathcal{P}re\mathcal{L}ie,-)$-algebra structure by Theorem \ref{braces}.
\end{proof}

In the symmetric context, we can recover an analogue of this corollary for $\bigoplus_{r\geq 1}\mathcal{P}(r)^{\Sigma_r}$. However, the operations $-\langle-,\ldots,-\rangle$ do not preserve $\bigoplus_{r\geq 1}\mathcal{P}(r)^{\Sigma_r}$. We thus need to force the symmetry, and then to sum on every possible positions of $q_1,\ldots,q_n$ in order to retrieve a $\Gamma(\mathcal{P}re\mathcal{L}ie,-)$-algebra structure.

\begin{prop}\label{symopealg}
    Suppose that $\mathbb{K}$ is a field and let $\mathcal{P}$ be a symmetric dg operad such that $\mathcal{P}(0)=0$. Then $\mathcal{L}(\mathcal{P})=\bigoplus_{r\geq 1}\mathcal{P}(r)^{\Sigma_r}$ is endowed with a $\Gamma(\mathcal{P}re\mathcal{L}ie,-)$-algebra structure defined by
    $$p\{q_1,\ldots,q_n\}_{r_1,\ldots,r_n}=\sum_{\sigma\in Sh(r_1,\ldots,r_n)}\sum_{\substack{1\leq i_1<\cdots<i_r\leq u\\ \omega\in Sh_\ast(1,...,\underset{i_1}{\overline{s}_{\sigma^{-1}(1)}},\ldots,\underset{i_r}{\overline{s}_{\sigma^{-1}(r)}},\ldots,1)}}\pm \omega\cdot\gamma(p(1,\ldots,\underset{i_1}{\overline{q}_{\sigma^{-1}(1)}},\ldots,\underset{i_r}{\overline{q}_{\sigma^{-1}(r)}},\ldots,1)),$$

    \noindent for elements of homogeneous arity $p\in\mathcal{P}(m)^{\Sigma_m}, q_1\in\mathcal{P}(s_1)^{\Sigma_{s_1}},\ldots,q_n\in\mathcal{P}(s_n)^{\Sigma_{s_n}}$ and where we have set $r=\sum_i r_i$, $(\overline{q}_1,\ldots,\overline{q}_r)=(\underbrace{q_1,\ldots,q_1}_{r_1},\ldots,\underbrace{q_n,\ldots,q_n}_{r_n})$ and $(\overline{s}_1,\ldots,\overline{s}_r)=(\underbrace{s_1,...,s_1}_{r_1},\ldots,\underbrace{s_n,\ldots,s_n}_{r_n})$. The sign is induced by the commutation of $\overline{q}_1,\ldots,\overline{q}_r$ to $\overline{q}_{\sigma^{-1}(r)},\ldots,\overline{q}_{\sigma^{-1}(1)}$. We also set $p\{q_1,\ldots,q_n\}_{r_1,\ldots,r_n}=0$ if $u<r_1+\cdots+r_n$. The weighted brace operations are then extended to the sum $\bigoplus_{r\geq 1}\mathcal{P}(r)^{\Sigma_r}$ by using Formula $(v)$ of Theorem \ref{Cesaro}.
\end{prop}

\begin{proof}
     We first prove that these operations preserve $\mathcal{L}(\mathcal{P})$. Let $p\in\mathcal{P}(m)^{\Sigma_m}, q_1\in\mathcal{P}(s_1)^{\Sigma_{s_1}},\ldots,q_n\in\mathcal{P}(s_n)^{\Sigma_{s_n}}$ and $r_1,\ldots,r_n\geq 0$. Notice that since we have
     $$p\{q_1,\ldots,q_n\}_{r_1,\ldots,r_n}=p\{q_1,\ldots,q_n,1\}_{r_1,\ldots,r_n,m-(r_1+\cdots+r_n)},$$ 
     
     \noindent we can suppose that $m=r_1+\cdots+r_n$. We then have
    $$p\{q_1,\ldots,q_n\}_{r_1,\ldots,r_n}=\sum_{\sigma\in Sh(r_1,\ldots,r_n)}\sum_{\omega\in Sh_\ast(\overline{s}_{\sigma^{-1}(1)},\ldots,\overline{s}_{\sigma^{-1}(r)})}\pm\omega\cdot\gamma(p(\overline{q}_{\sigma^{-1}(1)},\ldots,\overline{q}_{\sigma^{-1}(r)})).$$

    Let $\mu\in\Sigma_{s_1r_1+\cdots+s_nr_n}$. For a given $\sigma\in Sh(r_1,\ldots,r_n)$, we write $\mu\omega=\widetilde{\omega}\cdot\nu(\tau_1,\ldots,\tau_r)$ where $\nu\in\Sigma_{r},\tau_1\in\Sigma_{\overline{s}_{\sigma^{-1}(1)}},\ldots,\tau_r\in\Sigma_{\overline{s}_{\sigma^{-1}(r)}},\widetilde{\omega}\in Sh_\ast(\overline{s}_{(\nu\sigma)^{-1}(1)},\ldots,\overline{s}_{(\nu\sigma)^{-1}(r)})$. We also have set $\nu(\tau_1,\ldots,\tau_r)$ to be the composite of $\tau_1\oplus\cdots\oplus\tau_r$ with the corresponding blocks permutation given by $\nu\in\Sigma_r$. We obtain
    $$\mu\omega\cdot\gamma(p(\overline{q}_{\sigma^{-1}(1)},\ldots,\overline{q}_{\sigma^{-1}(r)}))=\widetilde{\omega}\cdot\gamma(p(\overline{q}_{(\nu\sigma)^{-1}(1)},\ldots,\overline{q}_{(\nu\sigma)^{-1}(r)})),$$

    \noindent as $p,q_1,\ldots,q_n$ are invariants. We now write $\nu\sigma=\widetilde{\sigma}\cdot(\widetilde{\tau_1}\oplus\cdots\oplus\widetilde{\tau_r})$ where $\widetilde{\sigma}\in Sh(r_1,\ldots,r_n),\widetilde{\tau_1}\in\Sigma_{r_1},\ldots,\widetilde{\tau_r}\in\Sigma_{r_n}$. We thus obtain
    $$\mu\omega\cdot\gamma(p(\overline{q}_{\sigma^{-1}(1)},\ldots,\overline{q}_{\sigma^{-1}(r)}))=\widetilde{\omega}\cdot\gamma(p(\overline{q}_{\widetilde{\sigma}^{-1}(1)},\ldots,\overline{q}_{\widetilde{\sigma}^{-1}(r)})).$$
    
    \noindent We thus have proved that
    $$\mu\cdot(p\{q_1,\ldots,q_n\}_{r_1,\ldots,r_n})=\sum_{\widetilde{\sigma}\in Sh(r_1,\ldots,r_n)}\sum_{\widetilde{\omega}\in Sh_\ast(\overline{s}_{\widetilde{\sigma}^{-1}(1)},\ldots,\overline{s}_{\widetilde{\sigma}^{-1}(r)})}\pm\widetilde{\omega}\cdot\gamma(p(\overline{q}_{\widetilde{\sigma}^{-1}(1)},\ldots,\overline{q}_{\widetilde{\sigma}^{-1}(r)}))=p\{q_1,\ldots,q_n\}_{r_1,\ldots,r_n}.$$

    \noindent The operations $-\{-,\ldots,-\}_{r_1,\ldots,r_n}$ then preserve $\mathcal{L}(\mathcal{P})$.\\
    
    We now prove formulas of Theorem \ref{prelie1}. We can immediately check that formulas $(i)-(v)$ are satisfied. The commutation with the differential is also satisfied since the operadic structure is compatible with the differential. It remains to prove formula $(vi)$ of Theorem \ref{prelie1}. We first note that the theorem holds if $\mathbb{K}=\mathbb{Q}$. Indeed, in that case, the trace map $Tr:\mathcal{S}(\mathcal{P}re\mathcal{L}ie,-)\longrightarrow\Gamma(\mathcal{P}re\mathcal{L}ie,-)$ induces an isomorphism of monads. We thus only need to prove that the $\Gamma(\mathcal{P}re\mathcal{L}ie,-)$-algebra structure is induced by a pre-Lie algebra structure. This presumed pre-Lie algebra structure is given by
    $$p\{q\}_1=\sum_{i=1}^m\sum_{\omega\in Sh_\ast(1,\ldots,\underset{i}{n},\ldots,1)}\omega\cdot(p\circ_i q),$$
    
    \noindent where $p\in\mathcal{P}(m)^{\Sigma_m}$ and $q\in\mathcal{P}(n)^{\Sigma_n}$. We then recover the pre-Lie algebra structure given in \cite[$\mathsection$5.3.16]{loday}. We now need to prove that the operations $-\{-,\ldots,-\}_{1,\ldots,1}$ coincide with the symmetric braces $-\{-,\ldots,-\}$ induced by the pre-Lie operation (see Definition \ref{prelie}). It is equivalent to prove the identity
    $$p\{q_1,\ldots,q_{n+1}\}_{1,\ldots,1}=p\{q_1,\ldots,q_n\}_{1,\ldots,1}\{q_{n+1}\}_1-\sum_{k=1}^n\pm p\{q_1,\ldots,q_k\{q_{n+1}\}_{1},\ldots,q_n\}_{1,\ldots,1}.$$

    \noindent This follows from the associativity of the operadic composition. More precisely, the term $p\{q_1,\ldots,q_n\}_{1,\ldots,1}\{q_{n+1}\}_1$ is composed of two types of operadic composition. Either $q_{n+1}$ is in the same level as $q_1,\ldots,q_n$, which will give exactly $p\{q_1,\ldots,q_{n+1}\}_{1,\ldots,1}$ by definition, or $q_{n+1}$ will be attached to one of $q_1,\ldots,q_n$. These last terms are removed in order to retrieve $p\{q_1,\ldots,q_{n+1}\}_{1,\ldots,1}$.\\

    We now prove the general case. Consider elements $p,q_1,\ldots,q_n,f_1,\ldots,f_m$ in $\bigoplus_{r\geq 1}\mathcal{P}(r)^{\Sigma_r}$ which are homogeneous in degrees and in arities, and $r_1,\ldots,r_n,s_1,\ldots,s_m\geq 0$. We need to compute $p\{q_1,\ldots,q_n\}_{r_1,\ldots,r_n}\{f_1,\ldots,f_m\}_ {s_1,\ldots,s_m}$ and to find the right hand-side of Theorem \ref{prelie1}. We consider the symmetric sequence $M_\mathbb{K}$, defined over any ring $\mathbb{K}$, spanned by abstract variables $P,Q_1,\ldots,Q_n,F_1,\ldots,F_m$ of the same arities and degrees as $p,q_1,\ldots,q_n,f_1,\ldots,f_m$ and endowed with a trivial action of the symmetric groups. We have an obvious morphism of symmetric sequences $M_\mathbb{K}\longrightarrow\mathcal{P}$ which sends $P$ to $p$, the $Q_i$'s to the $q_i$'s and the $F_j$'s to the $f_j$'s. We thus have a unique morphism of operads $\mathcal{F}(M_\mathbb{K})\longrightarrow\mathcal{P}$ which extends the morphism $M_\mathbb{K}\longrightarrow\mathcal{P}$, where $\mathcal{F}(M_\mathbb{K})$ is the free operad generated by the symmetric sequence $M_\mathbb{K}$. Because the presumed $\Gamma(\mathcal{P}re\mathcal{L}ie,-)$-algebra structure is written in terms of the operadic composition, if formula $(vi)$ of Theorem \ref{prelie1} holds for $\mathcal{F}(M_\mathbb{K})$, then it holds also for $\mathcal{P}$.\\

    We prove first that the formula is satisfied for $\mathcal{F}(M_\mathbb{Z})$. Since the morphism of rings $\mathbb{Z}\hookrightarrow\mathbb{Q}$ induces an injective morphism of operads $\mathcal{F}(M_\mathbb{Z})\hookrightarrow\mathcal{F}(M_\mathbb{Q})$ and since the relation $(vi)$ of Theorem \ref{prelie1} is true in $\mathcal{F}(M_\mathbb{Q})$, then it is also satisfied in $\mathcal{F}(M_\mathbb{Z})$. By using the morphism $\mathbb{Z}\longrightarrow\mathbb{K}$ which gives rise to a morphism of operads $\mathcal{F}(M_\mathbb{Z})\longrightarrow\mathcal{F}(M_\mathbb{K})$, we find that it is also satisfied in $\mathcal{F}(M_\mathbb{K})$.\\

    We thus have weighted braces operation on $\bigoplus_{r\geq 1}\mathcal{P}(r)^{\Sigma_r}$. Since we work on a field $\mathbb{K}$, by Theorem \ref{prelie2}, it implies that this structure endows $\bigoplus_{r\geq 1}\mathcal{P}(r)^{\Sigma_r}$ with a structure of a $\Gamma(\mathcal{P}re\mathcal{L}ie,-)$-algebra.
    \end{proof}

\begin{cor}\label{prod}
    Suppose that $\mathbb{K}$ is a field and let $\mathcal{P}$ be a dg operad such that $\mathcal{P}(0)=0$ and $\mathcal{P}(1)=\mathbb{K}$. Then $\prod_{r\geq 2}\mathcal{P}(r)^{\Sigma_r}$ is a complete $\Gamma(\mathcal{P}re\mathcal{L}ie,-)$-algebra.
\end{cor}

\begin{proof}
    Let $\mathcal{L}(\mathcal{P})=\bigoplus_{r\geq 2}\mathcal{P}(r)^{\Sigma_r}$. Note that $\mathcal{L}(\mathcal{P})$ is a sub $\Gamma(\mathcal{P}re\mathcal{L}ie,-)$-algebra of $\bigoplus_{r\geq 1}\mathcal{P}(r)^{\Sigma_r}$. We have a filtration on it given by $F_k\mathcal{L}(\mathcal{P})=\bigoplus_{r\geq k+1}\mathcal{P}(r)^{\Sigma_r}$ which is preserved by the weighted brace operations. The completion with respect to this filtration is exactly $\prod_{r\geq 2}\mathcal{P}(r)^{\Sigma_r}$. By the remark after Definition \ref{filteredgamma}, we have weighted brace operations on $\prod_{r\geq 2}\mathcal{P}(r)^{\Sigma_r}$. Since we work over a field, this endows $\prod_{r\geq 2}\mathcal{P}(r)^{\Sigma_r}$ with a structure of a $\Gamma(\mathcal{P}re\mathcal{L}ie,-)$-algebra by Theorem \ref{prelie2}.
\end{proof}

\subsection{The gauge group}\label{sec:22}

We can now define an analogue of the circular product given in \cite{dotsenko} using the weighted brace operations. Before doing so, we adopt the following notations. Let $L$ be a complete $\Gamma(\mathcal{P}re\mathcal{L}ie,-)$-algebra and $L_{+}=\mathbb{K}1\oplus L$. We extend the weighted braces $-\{-,\ldots,-\}_{r_1,\ldots,r_n}:L^{\times n+1}\longrightarrow L$ on $L_+\times L^{\times n}$ by setting 
$$1\{y_{1},\ldots,y_{n}\}_{r_1,\ldots,r_n}=\left\{\begin{array}{ll} y_{i} & \text{if}\ r_i=1\ \text{and}\ \forall k\neq i,r_k=0\ \\ 0 & \text{if}\ r_1+\cdots+r_n>1\end{array}\right..$$

\noindent for every $y_1,\ldots,y_n\in L$. We can check that all the formulas from Theorem \ref{Cesaro} are still satisfied if we take $x\in L_+$.

\begin{defi}
    Let $\alpha\in L_+$ and $\mu\in L^{0}$. We set
    $$\alpha\circledcirc (1+\mu)=\sum_{n=0}^{+\infty}\alpha\{\mu\}_{n}.$$
\end{defi}

Note that this quantity is well defined since $L$ is complete, and because $1\{y\}_n=0$ as soon as $n\geq 2$.\\

By applying this definition in the case $\mathbb{Q}\subset\mathbb{K}$ and using the weighted braces given by Remark \ref{rembraces}, we retrieve the usual circular product given in \cite{dotsenko}.

\begin{remarque}
    One can see that we have $1\circledcirc(1+\mu)=1+\mu=(1+\mu)\circledcirc 1$ so that $1$ is a unit element for $\circledcirc$. We thus have
    $$\forall \mu,\nu\in L^{0},(1+\mu)\circledcirc(1+\nu)=1+\nu+\sum_{n=0}^{+\infty}\mu\{\nu\}_{n},$$
    \noindent which shows that $\circledcirc$ preserves $1+L^0$.
\end{remarque}

\begin{lm}\label{ass}
The circular product $\circledcirc$ is associative, in the sense that for all $\alpha\in L_+$ and $\mu,\nu\in L^{0}$,
$$(\alpha\circledcirc(1+\mu))\circledcirc(1+\nu)=\alpha\circledcirc((1+\mu)\circledcirc(1+\nu)).$$
\end{lm}

\begin{proof}
    Let $\alpha\in L_+$ and $\mu,\nu\in L^{0}$. We first have
\begin{center}
$\begin{array}{lll}
    (\alpha\circledcirc (1+\mu))\circledcirc (1+\nu) & = & \displaystyle\left(\sum_{n=0}^{+\infty}\alpha\{\mu\}_{n}\right)\circledcirc (1+\nu)  \\
    & = &\displaystyle \sum_{n,p=0}^{+\infty}\alpha\{\mu\}_{n}\{\nu\}_{p}.  
\end{array}$
\end{center}

On the other hand, we have
\begin{center}
$\begin{array}{lll}
    \alpha\circledcirc((1+\mu)\circledcirc (1+\nu)) & = & \displaystyle \alpha\circledcirc\left(1+\nu+\sum_{n=0}^{+\infty}\mu\{\nu\}_{n}\right)  \\
    & = &\displaystyle \sum_{p=0}^{+\infty}\alpha\left\{\nu+\sum_{n=0}^{+\infty}\mu\{\nu\}_{n}\right\}_{p}.
\end{array}$
\end{center}

We thus need to prove that
$$\sum_{n,p=0}^{+\infty}\alpha\{\mu\}_{n}\{\nu\}_{p}=\sum_{p=0}^{+\infty}\alpha\left\{\nu+\sum_{n=0}^{+\infty}\mu\{\nu\}_{n}\right\}_{p}.$$

To prove this identity, we use formula $(vi)$ of Theorem \ref{Cesaro}:

$$\alpha\{\mu\}_{n}\{\nu\}_{p}=\sum_{p=\beta+\sum_{i=1}^{n}\alpha^{i}}\frac{1}{n!}\alpha\{\mu\{\nu\}_{\alpha^{1}},\ldots,\mu\{\nu\}_{\alpha^{n}},\nu\}_{1,\ldots,1,\beta},$$

\noindent which gives

$$\sum_{n,p=0}^{+\infty}\alpha\{\mu\}_{n}\{\nu\}_{p}=\sum_{n=0}^{+\infty}\sum_{p=0}^{+\infty}\sum_{\beta=0}^{p}\sum_{p-\beta=\sum_{i=1}^{n}\alpha^{i}}\frac{1}{n!}\alpha\{\mu\{\nu\}_{\alpha^{1}},\ldots,\mu\{\nu\}_{\alpha^{n}},\nu\}_{1,\ldots,1,\beta}.$$

In this sum, because of the symmetry, some terms occur several times. For a given $p$ and $\beta$, we count the number of partitions of $p-\beta=\alpha^{1}+\cdots+\alpha^{n}$ of the particular form $r_{1}\widetilde{\alpha^{1}}+\cdots+r_{q}\widetilde{\alpha^{q}}$. We get $ n(\widetilde{\alpha^{1}},\ldots,\widetilde{\alpha^{n}})=\frac{n!}{r_{1}!\cdots r_{q}!}$ for this number. We then have
$$\frac{1}{r_{1}!\cdots r_{q}!}\alpha\{\mu\{\nu\}_{\widetilde{\alpha}^{1}},\ldots,\mu\{\nu\}_{\widetilde{\alpha}^{1}},\ldots,\mu\{\nu\}_{\widetilde{\alpha}^{q}},\ldots,\mu\{\nu\}_{\widetilde{\alpha}^{q}},\nu\}_{1,\ldots,1,\beta}=\alpha\{\mu\{\nu\}_{\widetilde{\alpha}^{1}},\ldots,\mu\{\nu\}_{\widetilde{\alpha}^{q}},\nu\}_{r_{1},\ldots,r_{q},\beta}.$$

We conclude by formula $(v)$ of Theorem \ref{Cesaro}.
\end{proof}

We now need to find an explicit inverse for a given element $1-\mu$ with $\mu\in L^{0}$.

\begin{defi}
    Let $t$ be a non-labeled tree with $n$ vertices and $\mu\in L^0$. We set
    $$\mathcal{O}t(\mu)=\gamma(\mathcal{O}t(\mu^{\otimes n})),$$

    \noindent for some choice of labeling of $t$.
\end{defi}

Note that because $\mathcal{O}$ is $\Sigma$-invariant, this quantity does not depend on the choice of a labeling for $t$. For example, let $t$ be the non-labeled tree
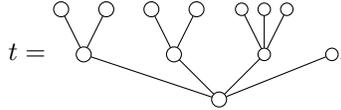
\begin{equation*}
    t=
    \begin{tikzpicture}[baseline={([yshift=-.5ex]current bounding box.center)},scale=0.6]
    \node[draw,circle,scale=0.6] (i) at (0,0) {};
    \node[draw,circle,scale=0.6] (1) at (-3,1) {};
    \node[draw,circle,scale=0.6] (1bis) at (-3.5,2) {};
    \node[draw,circle,scale=0.6] (1bbis) at (-2.5,2) {};
    \node[draw,circle,scale=0.6] (2) at (-1,1) {};
    \node[draw,circle,scale=0.6] (2bis) at (-1.5,2) {};
    \node[draw,circle,scale=0.6] (2bbis) at (-0.5,2) {};
    \node[draw,circle,scale=0.5] (3) at (1,1) {};
    \node[draw,circle,scale=0.5] (3bis) at (1.5,2) {};
    \node[draw,circle,scale=0.5] (3bbis) at (0.5,2) {};
    \node[draw,circle,scale=0.5] (3bbbis) at (1,2) {};
    \node[draw,circle,scale=0.5] (4) at (2.5,1) {};
    \draw (1) -- (1bis);
    \draw (1) -- (1bbis);
    \draw (2) -- (2bis);
    \draw (2) -- (2bbis);
    \draw (3) -- (3bis);
    \draw (3) -- (3bbis);
    \draw (3) -- (3bbbis);
    \draw (4) -- (i);
    \draw (2) -- (i);
    \draw (i) -- (1);
    \draw (i) -- (3);
    \end{tikzpicture}.\ \ 
\end{equation*}

\noindent Then
$$\mathcal{O}t(\mu)=\mu\{\mu\{\mu\}_{2},\mu\{\mu\}_{3},\mu\}_{2,1,1}.$$

\begin{lm}\label{inverse}
    For every $\mu\in L^{0}$, the element $1-\mu$ has an inverse in $1+L^0$ for the circular product $\circledcirc$ given by
    $$(1-\mu)^{\circledcirc -1}=1+\sum_{t\in rRT^{\ast}}\mathcal{O}t(\mu),$$

    \noindent where $rRT^{\ast}$ is the set of trees without any labeling and with at least one vertex.
\end{lm}

\begin{proof}
    We first see that this defines a right-inverse for $1-\mu$. Indeed, we first have that
    $$(1-\mu)\circledcirc\left(1+\sum_{t\in rRT^{*}}\mathcal{O}t(\mu)\right)=1+\sum_{t\in rRT^{*}}\mathcal{O}t(\mu)-\sum_{k=0}^{+\infty}\mu\left\{\sum_{t\in rRT^{*}}\mathcal{O}t(\mu)\right\}_{k}.$$

    \noindent Then, as every $t\in rRT^{\ast}$ can be uniquely described by its root and branches, we have that every term in the first sum at the right hand side can be uniquely described by an element from the second sum, and vice versa. Formula $(v)$ from Theorem \ref{Cesaro} thus give the result.\\

    We now need to prove that it is a left inverse, which is slightly more difficult. We compute
    $$\left(1+\sum_{t\in rRT^{\ast}}\mathcal{O}t(\mu)\right)\circledcirc (1-\mu)=1-\mu+\sum_{t\in rRT^{\ast}}\mathcal{O}t(\mu)+\sum_{k\geq 1}\left(\sum_{t\in rRT^{*}}\mathcal{O}t(\mu)\right)\{-\mu\}_{k}.$$

    We focus on one term $\mathcal{O}t(\mu)$ from the first sum, for some tree $t\in rRT^{\ast}$. Recall that a vertex of $t$ is called a \textit{leaf} if it is not the root of $t$ and if it is connected to one and only one other vertex in $t$. We denote by $m_t$ the number of leaves.\\
    
    If $m_t=0$, then $t$ is the trivial tree: $\mathcal{O}t(\mu)=\mu$. This term does not appear in the second sum (because $k\geq 1$) and vanishes with $-\mu$.\\
    
    If $m_t\neq 0$, the idea is to fix a number $1\leq k\leq m_t$, and to see which trees we can obtain if we remove $k$ leaves of $t$. These trees will occur in the second sum and give $(-1)^{k}\mathcal{O}t(\mu)$ by adding $k$ copies of $-\mu$.\\

    Let $X_{t}$ be the set of leaves of $t$. Let $X_{t,k}$ be the set of non ordered subsets of $X_{t}$ with $k$ elements. When we remove $k$ leaves, we need to take into account that we can obtain the same tree by removing a different non ordered set of $k$ leaves. For example, if we take the previous tree
    \begin{equation*}
    t=
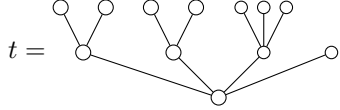
    \begin{tikzpicture}[baseline={([yshift=-.5ex]current bounding box.center)},scale=0.6]
    \node[draw,circle,scale=0.6] (i) at (0,0) {};
    \node[draw,circle,scale=0.6] (1) at (-3,1) {};
    \node[draw,circle,scale=0.6] (1bis) at (-3.5,2) {};
    \node[draw,circle,scale=0.6] (1bbis) at (-2.5,2) {};
    \node[draw,circle,scale=0.6] (2) at (-1,1) {};
    \node[draw,circle,scale=0.6] (2bis) at (-1.5,2) {};
    \node[draw,circle,scale=0.6] (2bbis) at (-0.5,2) {};
    \node[draw,circle,scale=0.5] (3) at (1,1) {};
    \node[draw,circle,scale=0.5] (3bis) at (1.5,2) {};
    \node[draw,circle,scale=0.5] (3bbis) at (0.5,2) {};
    \node[draw,circle,scale=0.5] (3bbbis) at (1,2) {};
    \node[draw,circle,scale=0.5] (4) at (2.5,1) {};
    \draw (1) -- (1bis);
    \draw (1) -- (1bbis);
    \draw (2) -- (2bis);
    \draw (2) -- (2bbis);
    \draw (3) -- (3bis);
    \draw (3) -- (3bbis);
    \draw (3) -- (3bbbis);
    \draw (4) -- (i);
    \draw (2) -- (i);
    \draw (i) -- (1);
    \draw (i) -- (3);
    \end{tikzpicture}\ \
\end{equation*}

\noindent and if we look at the first branch, removing the vertex at the left gives the same tree as removing the vertex at the right.\\

Let $t_{k}^{1},\ldots,t_{k}^{p_{k}}$ be all the different trees that we can get from $t$ by removing $k$ leaves. We denote by $X_{t,k}^{t_{k}^{i}}$ the subset of $X_{t,k}$ formed by all the vertices that lead to $t_{k}^{i}$ when removing them from $t$. We then have a disjoint union $\displaystyle X_{t,k}=\bigsqcup_{i=1}^{p_{k}} X_{t,k}^{t_{k}^{i}}$.\\

Each terms $\mathcal{O}t_{k}^{i}(\mu)\{-\mu\}_{k}$ will then give, among other terms, $(-1)^{k}Card(X_{t,k}^{t_{k}^{i}}) \mathcal{O}t(\mu)$. When we take the sum over $i$, we obtain $\displaystyle (-1)^{k}Card(X_{t,k})\mathcal{O}t(\mu)=(-1)^{k}\binom{m_t}{k}\mathcal{O}t(\mu)$. By taking the sum over $k$, we therefore obtain $-\mathcal{O}t(\mu)$ which vanishes with $\mathcal{O}t(\mu)$ given by the first sum.
\end{proof}

From Lemma \ref{ass} and Lemma \ref{inverse}, we deduce assertion $(i)$ of Theorem \ref{theoremA}:

\begin{thm}
    The triple $G=(1+L^{0},\circledcirc,1)$ is a group called the gauge group of $L$.\\
\end{thm}

\subsection{Maurer-Cartan elements and the Deligne groupoid}\label{sec:23}



We now aim to prove assertion $(ii)$ of Theorem \ref{theoremA}. We first make explicit the definition of the Maurer-Cartan set.

\begin{defi}
    Let $L$ be a dg $\Gamma(\mathcal{P}re\mathcal{L}ie,-)$-algebra. A given $\alpha\in L^{1}$ is a {\normalfont Maurer-Cartan element} if it satisfies the {\normalfont Maurer-Cartan equation}:
    $$d(\alpha)+\alpha\{\alpha\}_{1}=0.$$

    We let $\mathcal{MC}(L)$ to be the set of all Maurer-Cartan elements of $L$.
\end{defi}

\begin{remarque}
    In the case $\mathbb{Q}\subset\mathbb{K}$, we retrieve the classical definition:
    $$d(\alpha)+\frac{1}{2}[\alpha,\alpha]=0,$$

    \noindent written with the dg Lie algebra structure on $L$.
\end{remarque}

As in the case of characteristic zero, we expect the gauge group to act on the Maurer-Cartan set. Before seeing that, we define a new operation.

\begin{defi}
    Let $\alpha\in L_+,\beta\in L$ and $1+\mu\in G$. We set
    $$\alpha\circledcirc(1+\mu;\beta)=\sum_{n=0}^{+\infty}\alpha\{\mu,\beta\}_{n,1}.$$
\end{defi}

\begin{lm}
    We have the following identities:

    \begin{center}
$\begin{array}{lll}
(\alpha\circledcirc(1+\mu))\{\beta\}_{1} & = & \alpha\circledcirc(1+\mu;\beta+\mu\{\beta\}_{1}),\\
\alpha\{\beta\}_{1}\circledcirc(1+\mu) & = & \alpha\circledcirc(1+\mu;\beta\circledcirc(1+\mu)),\\
d(\alpha\circledcirc(1+\mu)) & = & d(\alpha)\circledcirc(1+\mu)+(-1)^{|\alpha|}\alpha\circledcirc(1+\mu;d(\mu)).
\end{array}$
\end{center}
\end{lm}

\begin{proof}
    By applying formula $(vi)$ of Theorem \ref{Cesaro}, we find that
\begin{center}
    $\begin{array}{lll}
        (\alpha\circledcirc(1+\mu))\{\beta\}_{1} & = & \displaystyle\sum_{n=0}^{+\infty}\alpha\{\mu\}_{n}\{\beta\}_{1} \\
         & = & \displaystyle\sum_{n=0}^{+\infty}\alpha\{\mu,\beta\}_{n,1}+\sum_{n=1}^{+\infty}\alpha\{\mu,\mu\{\beta\}_{1}\}_{n-1,1}\\
         & = & \displaystyle\sum_{n=0}^{+\infty}\alpha\{\mu,\beta+\mu\{\beta\}_{1}\}_{n,1}\\
         & = & \alpha\circledcirc(1+\mu;\beta+\mu\{\beta\}_{1}),
    \end{array}$
\end{center}

\noindent as well as
\begin{center}
    $\begin{array}{lll}
        \alpha\{\beta\}_{1}\circledcirc(1+\mu) & = &\displaystyle\sum_{m=0}^{+\infty}\alpha\{\beta\}_{1}\{\mu\}_{m}\\
        & = & \displaystyle\sum_{p,q=0}^{+\infty}\alpha\{\beta\{\mu\}_{p},\mu\}_{1,q}\\
        & = & \alpha\circledcirc(1+\mu;\beta\circledcirc(1+\mu)).
    \end{array}$
\end{center}

Finally, by using the compatibility of $d$ with weighted braces, we obtain,
\begin{center}
    $\begin{array}{lll}
        d(\alpha\circledcirc(1+\mu)) & = & \displaystyle\sum_{n=0}^{+\infty}d(\alpha)\{\mu\}_{n}+(-1)^{|\alpha|}\sum_{n=1}^{+\infty}\alpha\{\mu,d(\mu)\}_{n-1,1}\\
         & = & \displaystyle d(\alpha)\circledcirc(1+\mu)+(-1)^{|\alpha|}\alpha\circledcirc(1+\mu;d(\mu)),\\
    \end{array}$\\
\end{center}

\noindent which concludes the proof of the lemma.
\end{proof}

We can now prove assertion $(ii)$ of Theorem \ref{theoremA}.

\begin{thm}
    The gauge group $G$ acts on the Maurer-Cartan set $\mathcal{MC}(L)$ by
    $$(1+\mu)\cdot\alpha=(\alpha+\mu\{\alpha\}_{1}-d(\mu))\circledcirc(1+\mu)^{\circledcirc-1}$$
    \noindent for all $(1+\mu)\in G$ and $\alpha\in\mathcal{MC}(L)$.
\end{thm}



\begin{proof}
    We first need to prove that $\beta=(1+\mu)\cdot\alpha$ is indeed a Maurer-Cartan element. For this, we first remark that applying $d$ on each side of the equality $d(\mu)=\alpha+\mu\{\alpha\}_{1}-\beta\circledcirc(1+\mu)$, and by using that $d(\alpha)=-\alpha\{\alpha\}_{1}$ and the previous lemma, we have
$$d(\beta)\circledcirc(1+\mu)=-\alpha\{\alpha\}_{1}-\mu\{\alpha\{\alpha\}_{1}\}_{1}+d(\mu)\{\alpha\}_{1}+\beta\circledcirc (1+\mu;d(\mu)).$$

Moreover, again by the previous lemma, we have

\begin{center}
    $\begin{array}{lll}
        d(\mu)\{\alpha\}_{1} & = & \alpha\{\alpha\}_{1}+\mu\{\alpha\}_{1}\{\alpha\}_{1}-\beta\circledcirc(1+\mu)\{\alpha\}_{1}  \\
        & = & \alpha\{\alpha\}_{1}+\mu\{\alpha\{\alpha\}_{1}\}_{1}+\mu\{\alpha,\alpha\}_{1,1}-\beta\circledcirc(1+\mu;\alpha)-\beta\circledcirc(1+\mu;\mu\{\alpha\}_{1}).\\
    \end{array}$
\end{center}

Then
$$d(\beta)\circledcirc(1+\mu)=\beta\circledcirc(1+\mu;d(\mu))-\beta\circledcirc(1+\mu;\alpha)-\beta\circledcirc(1+\mu;\mu\{\alpha\}_{1})+\mu\{\alpha,\alpha\}_{1,1}.$$

\noindent We note that $\mu\{\alpha,\alpha\}_{1,1}=0$. Indeed, if we take the notations of the proof of Theorem \ref{prelie1}, we have that $\mu\{\alpha,\alpha\}_{1,1}=\Gamma(\mathcal{P}re\mathcal{L}ie,\psi_{\mu,\alpha,\alpha})(\mathcal{O}F_2(e,e_1,e_2))$ where $|e|=0$ and $|e_1|=|e_2|=1$ are formal elements. We explicitly have that
$$\mathcal{O}F_2(e,e_1,e_2)=F_2(e,e_1,e_2)+((12).F_2)(e_1,e,e_2)-((13).F_2)(e_2,e_1,e)$$
$$-F_2(e,e_2,e_1)-((12).F_2)(e_2,e,e_1)+((13).F_2)(e_1,e_2,e).$$

\noindent Applying $\psi_{\mu,\alpha,\alpha}$ will then give $\mu\{\alpha,\alpha\}_{1,1}=0$. We then finally have
$$d(\beta)\circledcirc(1+\mu)=-\beta\circledcirc(1+\mu;\beta\circledcirc(1+\mu)).$$

By the previous lemma, this gives
$$d(\beta)\circledcirc(1+\mu) = -\beta\{\beta\}_{1}\circledcirc(1+\mu)$$

\noindent and then $(d(\beta)+\beta\{\beta\}_{1})\circledcirc(1+\mu)=0$, that is to say $d(\beta)+\beta\{\beta\}_{1}=0$ by composing with $(1+\mu)^{\circledcirc-1}$ on the right. We thus have proved that $\beta\in\mathcal{MC}(L)$.\\

We now need to check that we have indeed an action of $G$ on $\mathcal{MC}(L)$. We have that $1+0$ acts trivially on $\mathcal{MC}(L)$, so we just need to prove that $((1+\nu)\circledcirc(1+\mu))\cdot\alpha=(1+\nu)\cdot((1+\mu)\cdot\alpha)$.\\

By hypothesis, we have the following equations:
$$d(\mu)=\alpha+\mu\{\alpha\}_{1}-\beta\circledcirc(1+\mu),$$
$$d(\nu)=\beta+\nu\{\beta\}_{1}-\gamma\circledcirc(1+\nu).$$

Let $\displaystyle 1+\lambda=(1+\nu)\circledcirc(1+\mu)=1+\mu+\nu\circledcirc(1+\mu)$. We compute:

\begin{center}
$\begin{array}{lll}
    \alpha+\lambda\{\alpha\}_{1}-\gamma\circledcirc(1+\lambda) & = & \displaystyle\alpha+\mu\{\alpha\}_{1}+\nu\circledcirc(1+\mu)\{\alpha\}_{1}+d(\nu)\circledcirc(1+\mu)\\
     &  & \displaystyle-\beta\circledcirc(1+\mu)-\nu\{\beta\}_{1}\circledcirc(1+\mu)\\
     & &\\
     & = & \displaystyle d(\mu)+d(\nu)\circledcirc(1+\mu)+\nu\circledcirc(1+\mu;\alpha)\\
     & &\displaystyle+\nu\circledcirc(1+\mu;\mu\{\alpha\}_{1}) -\nu\circledcirc(1+\mu;\beta\circledcirc(1+\mu))
\end{array}$
\end{center}

\noindent by the previous lemma. We then have

\begin{center}
$\begin{array}{lll}
    \alpha+\lambda\{\alpha\}_{1}-\gamma\circledcirc(1+\nu)\circledcirc(1+\mu) & = & d(\mu)+d(\nu)\circledcirc(1+\mu)+\nu\circledcirc(1+\mu;d(\mu))\\
     & = & d(\lambda),
\end{array}$
\end{center}

\noindent which proves the theorem.
\end{proof}





We end this section with the definition of the \textit{Deligne groupoid}.

\begin{propdef}
    Let $L$ be a complete $\Gamma(\mathcal{P}re\mathcal{L}ie,-)$-algebra. We let $\text{\normalfont Deligne}(L)$ to be the category with $\mathcal{MC}(L)$ as set of objects and $\displaystyle \text{\normalfont Mor}_{\text{\normalfont Deligne}(L)}(\alpha,\beta)=\{(1+\mu)\in G\ |\ (1+\mu)\cdot\alpha=\beta\}$.\\ Then $\text{\normalfont Deligne}(L)$ is a groupoid called the {\normalfont Deligne groupoid} of $L$.
\end{propdef}

\begin{proof}
    It is a corollary of the previous theorem.
\end{proof}

\subsection{An integral Goldman-Millson theorem}\label{sec:24}

We conclude this part with an analogue of the Goldman-Millson theorem. This theorem allows us to give a link between two particular groupoids when changing a dg Lie algebra $L$ to another one $\overline{L}$ which is quasi-isomorphic to $L$ (see \cite[$\mathsection$2.4]{goldman}).\\


Let $A$ be a local artinian $\textbf{K}$-algebra with maximal ideal $\mathfrak{m}_A$, where $\textbf{K}$ is the field of fractions of some noetherian integral domain $\mathbb{K}$. Let $L$ be a $\Gamma(\mathcal{P}re\mathcal{L}ie,-)$-algebra (without any convergence hypothesis). If $\otimes$ denotes the tensor product over $\mathbb{K}$, then $L\otimes A$ is also a $\Gamma(\mathcal{P}re\mathcal{L}ie,-)$-algebra with the following definitions:
\begin{center}
    $\begin{array}{rll}
         (L\otimes A)^{k} & = & L^{k}\otimes A,\\
\gamma(\mathcal{O}t(x_{1}\otimes a_{1},\ldots,x_{n}\otimes a_{n}))& = & \gamma(\mathcal{O}t(x_{1},\ldots,x_{n}))\otimes a_{1}\cdots a_{n},\\
d(x\otimes a) & = & d(x)\otimes a.\\
    \end{array}$
\end{center}

To retrieve our convergence hypothesis, we can consider the sub $\Gamma(\mathcal{P}re\mathcal{L}ie,-)$-algebra $L\otimes\mathfrak{m}_A$. This $\Gamma(\mathcal{P}re\mathcal{L}ie,-)$-algebra has a filtration given by
$$F_n(L\otimes\mathfrak{m}_A)=L\otimes\mathfrak{m}_A^n$$

\noindent which is $0$ for $n$ big enough, because $\mathfrak{m}_A$ is nilpotent. In particular, our series will be reduced to finite sums.\\

Let $\text{\normalfont Deligne}(L,A)=\text{\normalfont Deligne}(L\otimes\mathfrak{m}_A)$ the associated Deligne groupoid. As in \cite[$\mathsection$2.3]{goldman}, we remark that $\text{\normalfont Deligne}(-,-)$ defines a bifunctor such that, for all morphisms of $\Gamma(\mathcal{P}re\mathcal{L}ie,-)$-algebras $\varphi:L\longrightarrow\overline{L}$ and for all morphisms of algebras $\psi:A\longrightarrow\overline{A}$, we have the following diagram
\begin{center}
    $\begin{tikzcd}
        \text{\normalfont Deligne}(L,A)\arrow[r, "\varphi_{*}"]\arrow[d, "\psi_{*}"'] & \text{\normalfont Deligne}(\overline{L},A)\arrow[d, "\psi_{*}"]\\
        \text{\normalfont Deligne}(L,\overline{A})\arrow[r, "\varphi_{*}"'] & \text{\normalfont Deligne}(\overline{L},\overline{A})
    \end{tikzcd}$
\end{center}

\noindent which is commutative.\\

We can now prove Theorem \ref{theoremB}.

\begin{thm}
    Let $\mathbb{K}$ be a noetherian integral domain and $\textbf{K}$ its field of fractions. Let $L$ and $\overline{L}$ be two positively graded $\Gamma(\mathcal{P}re\mathcal{L}ie,-)$-algebras. Let $\varphi:L\longrightarrow\overline{L}$ be a morphism of $\Gamma(\mathcal{P}re\mathcal{L}ie,-)$-algebras such that $H^{0}(\varphi)$ and $H^{1}(\varphi)$ are isomorphisms, and $H^{2}(\varphi)$ is a monomorphism. Then for every local artinian $\textbf{K}$-algebras $A$, the induced functor $\varphi_{*}:\text{\normalfont Deligne}(L,A)\longrightarrow\text{\normalfont Deligne}(\overline{L},A)$ is an equivalence of groupoids.
\end{thm}

\begin{proof}
    The proof is close to the one given in \cite[$\mathsection$2.5-$\mathsection$2.11]{goldman}. By the same arguments as in \cite[$\mathsection$2.5]{goldman}, we are reduced to prove the following. Let $A$ be a local artinian $\textbf{K}$-algebra with maximal ideal $\mathfrak{m}_A$ and $\mathfrak{I}\subset A$ be an ideal such that $\mathfrak{I}\cdot\mathfrak{m}_A=0$. Then if the theorem holds for $A/\mathfrak{I}$, then it also holds  for $A$.\\

    Our first goal is to prove the same proposition given in \cite[$\mathsection$2.6]{goldman}, which constructs three obstruction maps $o_2, o_1$ and $o_0$. Let $\pi_\ast: L\otimes A\longrightarrow L\otimes A/\mathfrak{I}$ be the map induced by the canonical projection $\pi:A\longrightarrow A/\mathfrak{I}$.\\
    
    We first define $o_2:Obj(\text{\normalfont Deligne}(L,A/\mathfrak{I}))\longrightarrow H^2(L\otimes\mathfrak{I})$ which is such that $o_2(\omega)=0$ if and only if there exists $\widetilde{\omega}\in Obj(\text{\normalfont Deligne}(L,A))$ with $\pi_\ast(\widetilde{\omega})=\omega$.\\
    
    \noindent Let $\omega\in Obj(\text{\normalfont Deligne}(L,A/\mathfrak{I}))\subset\mathfrak{m}_A/\mathfrak{I}$. Let $\widetilde{\omega}\in L^1\otimes\mathfrak{m}_A$ be such that $\pi_\ast(\widetilde{\omega})=\omega$ and $\mathcal{Q}(\widetilde{\omega})=d(\widetilde{\omega})+\widetilde{\omega}\{\widetilde{\omega}\}_1$. Then $\pi_\ast(\mathcal{Q}(\widetilde{\omega}))=0$ to that $\mathcal{Q}(\widetilde{\omega})\in L^2\otimes\mathfrak{I}$. By using that $\mathfrak{I}\cdot\mathfrak{m}_A=0$, we obtain
    \begin{center}
    $\begin{array}{lll}
        d(\mathcal{Q}(\widetilde{\omega})) & = & d(\widetilde{\omega})\{\widetilde{\omega}\}_{1}-\widetilde{\omega}\{d(\widetilde{\omega})\}_{1} \\
         & = & -\widetilde{\omega}\{\widetilde{\omega}\}_{1}\{\widetilde{\omega}\}_{1}+\widetilde{\omega}\{\widetilde{\omega}\{\widetilde{\omega}\}_{1}\}_{1}\\
         & = & \widetilde{\omega}\{\widetilde{\omega},\widetilde{\omega}\}_{1,1}\\
         & = & 0.
    \end{array}$
\end{center}

\noindent This implies that $\mathcal{Q}(\widetilde{\omega})\in Z^2(L\otimes\mathfrak{I})$. Let $\widetilde{\omega}'\in L^1\otimes\mathfrak{m}_A$ be some other element such that $\pi_\ast(\widetilde{\omega}')=\omega$. In particular, $\widetilde{\omega}-\widetilde{\omega}'\in L^1\otimes\mathfrak{I}$. We then have, using again that $\mathfrak{I}\cdot\mathfrak{m}_A=0$,
\begin{center}
    $\begin{array}{lll}
        \mathcal{Q}(\widetilde{\omega}')-\mathcal{Q}(\widetilde{\omega}) & = & d(\widetilde{\omega}'-\widetilde{\omega})+(\widetilde{\omega}'-\widetilde{\omega})\{\widetilde{\omega}'-\widetilde{\omega}\}_{1}\\
         & & +\widetilde{\omega}\{\widetilde{\omega}'-\widetilde{\omega}\}_{1}+(\widetilde{\omega}'-\widetilde{\omega})\{\widetilde{\omega}\}_{1}\\
         & = & d(\widetilde{\omega}'-\widetilde{\omega}).\\
    \end{array}$
\end{center}

\noindent We then let $o_2(\omega)$ to be the class of $\mathcal{Q}(\widetilde{\omega})$ in $H^2(L\otimes\mathfrak{I})$. Suppose that $o_2(\omega)=0$. Then by definition, there exists some $\psi\in L^1\otimes\mathfrak{I}$ such that $\mathcal{Q}(\widetilde{\omega})=d(\psi)$. We can check that $\widetilde{\omega}':=\widetilde{\omega}-\psi\in Obj(\text{\normalfont Deligne}(L,A))$ and $\pi_\ast(\widetilde{\omega}')=\omega$. In the converse direction, we obviously have that if $\omega=\pi_\ast(\widetilde{\omega})$ with $\widetilde{\omega}\in Obj(\text{\normalfont Deligne}(L,A))$, then $o_2(\omega)=0$.\\

    We now prove the following analogue of the lemma given in \cite[$\mathsection$2.8]{goldman}. For all $\alpha\in L^{1}\otimes\mathfrak{m}_A,\eta\in L^{0}\otimes\mathfrak{m}_A$ and $u\in L^{0}\otimes\mathfrak{I}$, we have
    $$(1+u+\eta)\cdot\alpha=(1+\eta)\cdot\alpha-d(u).$$

    \noindent Let $\beta=(1+\eta)\cdot\alpha$. We have
    \begin{center}
    $\begin{array}{lll}
        (\beta-d(u))\circledcirc(1+u+\eta) & = & \displaystyle\sum_{n=0}^{+\infty}(\beta-d(u))\{u+\eta\}_{n}  \\
         & = & \displaystyle\sum_{n=0}^{+\infty}\sum_{k=0}^{n}(\beta-d(u))\{u,\eta\}_{k,n-k}.\\
        \end{array}$
        \end{center}
    \noindent Since $\mathfrak{I}\cdot\mathfrak{m}_A=0$, and because $u\in L^1\otimes\mathfrak{I}$ and $\eta\in L^0\otimes\mathfrak{m}_A$, the terms with $n\neq 0$ and $k\neq 0$ are $0$ by definition of the $\Gamma(\mathcal{P}re\mathcal{L}ie,-)$-algebra structure in $L\otimes A$. We then have
    \begin{center}
    $\begin{array}{lll}
         (\beta-d(u))\circledcirc(1+u+\eta) & = & \displaystyle\beta-d(u)+\sum_{n=1}^{+\infty}\beta\{\eta\}_{n}\\
         & = & \beta\circledcirc(1+\eta)-d(u)\\
         & = & \alpha+\eta\{\alpha\}_{1}-d(\eta)-d(u)\\
         & = & \alpha+(u+\eta)\{\alpha\}_{1}-d(u+\eta)\\
    \end{array}$
\end{center}

\noindent which finally gives
$$(1+\eta)\cdot\alpha-d(u)=(\alpha+(u+\eta)\{\alpha\}_1-d(u+\eta))\circledcirc (1+u+\eta)^{\circledcirc-1}=(1+u+\eta)\cdot\alpha.$$

Let $\xi\in Obj(\text{\normalfont Deligne}(L,A/\mathfrak{I}))$. We let $\pi_\ast^{-1}(\xi)$ to be the category whose objects are elements $\omega\in Obj(\text{\normalfont Deligne}(L,A))$ such that $\pi_\ast(\omega)=\xi$, and with morphisms the gauge group elements $\gamma\in \text{\normalfont Deligne}(L,A)$ such that $\pi_\ast(\gamma)=1$. We now construct a map $o_1:Obj(\pi_\ast^{-1}(\xi))\times Obj(\pi_\ast^{-1}(\xi))\longrightarrow Z^1(L\otimes\mathfrak{I})$ such that $o_1(\alpha,\beta)=0$ if and only if there exists a morphism $\gamma$ in $\pi_\ast^{-1}(\xi)$ such that $\gamma(\alpha)=\beta$.\\

\noindent Let $\eta\in Z^1(L\otimes\mathfrak{I})$. For every $\alpha\in Obj(\pi_\ast^{-1}(\xi))$, we have that $\mathcal{Q}(\alpha+\eta)=\mathcal{Q}(\alpha)+\alpha\{\eta\}_1+\eta\{\eta\}_1=0$, since $\mathfrak{I}\cdot\mathfrak{m}_A=0$. We then have that $\alpha+\eta\in Obj(\text{\normalfont Deligne}(L,A))$, and $\pi_\ast(\alpha+\eta)=\xi$. We thus have an action of the group $Z^1(L\otimes\mathfrak{I})$ on the set $Obj(\pi_\ast^{-1}(\xi))$. This action is simply transitive. Indeed, let $\alpha,\beta\in Obj(\text{\normalfont Deligne}(L,A))$ be such that $\pi_\ast(\alpha)=\pi_\ast(\beta)=\xi$. Then $\alpha-\beta\in L^1\otimes\mathfrak{I}$ and $d(\alpha-\beta)=\mathcal{Q}(\alpha)-\mathcal{Q}(\beta)=0$. The element $\eta:=\alpha-\beta$ is then an element of $Z^1(L\otimes\mathfrak{I})$ whose action on $\beta$ is $\alpha$. Since $\alpha-\beta=0$ if and only if $\alpha=\beta$, the action is indeed simply transitive. We then can set $o_1(\alpha,\beta)$ to be the class of $\alpha-\beta$ in $H^1(L\otimes\mathfrak{I})$.\\

\noindent Now, if there exists some element $1+u\in G$ in the gauge group of $L\otimes \mathfrak{m}_A$ such that $(1+u)\cdot\alpha=\beta$ and $\pi_\ast(1+u)=1$, we then have, according to the analogue of the lemma given in \cite[$\mathsection$2.8]{goldman}, that $\beta=(1+u)\cdot\alpha=\alpha-d(u)$ so that $o_1(\alpha,\beta)=0$. In the converse direction, if $o_1(\alpha,\beta)=0$, then there exists $u\in L^0\otimes\mathfrak{I}$ such that $\alpha-\beta=d(u)$. By using the lemma again, we find that $(1+u)\cdot\alpha=\beta$.\\

Let $\widetilde{\alpha},\widetilde{\beta}\in Obj(\text{\normalfont Deligne}(L,A))$ and $\alpha=\pi_\ast(\widetilde{\alpha}),\beta=\pi_\ast(\widetilde{\beta})$ be such that there exists an element $1+u$ of the gauge group of $L\otimes\mathfrak{m}_A/\mathfrak{I}$ with $(1+u)\cdot\alpha=\beta$. Let $\pi_\ast^{-1}(1+u)$ be the set of gauge group elements $1+\widetilde{u}$ in $L\otimes\mathfrak{m}_A$ such that $(1+\widetilde{u})\cdot\widetilde{\alpha}=\widetilde{\beta}$ and $\pi_\ast(1+\widetilde{u})=1+u$. We finally construct a map $o_0:\pi_\ast^{-1}(1+u)\times\pi_\ast^{-1}(1+u)\longrightarrow H^0(L\otimes\mathfrak{I})$ which satisfies the following. For every $1+\widetilde{u},1+\widetilde{u}'\in\pi_\ast^{-1}(1+u)$, we have that $o_0(1+\widetilde{u},1+\widetilde{u}')=0$ if and only if $\widetilde{u}=\widetilde{u}'$.\\

\noindent We define a simply transitive action of $H^0(L\otimes\mathfrak{I})$ on the set $\pi_\ast^{-1}(1+u)$. Let $1+\widetilde{u}\in\pi_\ast^{-1}(1+u)$ and let $w\in L\otimes\mathfrak{I}$. We have that
$$(1+\widetilde{u}+w)\cdot\widetilde{\alpha}=(1+\widetilde{u})\cdot\widetilde{\alpha}-d(w)=\widetilde{\beta}-d(w).$$

\noindent Therefore, if $w\in H^0(L\otimes\mathfrak{I})=Z^0(L\otimes\mathfrak{I})$, then $(1+\widetilde{u}+w)\cdot\widetilde{\alpha}=\widetilde{\beta}$. This then defines an action of $H^0(L\otimes\mathfrak{I})$ on $\pi_\ast^{-1}(1+u)$. Let $1+\widetilde{u}'\in\pi_\ast^{-1}(1+u)$. We set $w=o_0(1+\widetilde{u},1+\widetilde{u}'):=\widetilde{u}-\widetilde{u}'\in L\otimes\mathfrak{I}$. Then
$$d(w)=d(o_0(1+\widetilde{u},1+\widetilde{u}'))=(1+\widetilde{u}')\cdot\widetilde{\alpha}-(1+\widetilde{u}'+w)\cdot\widetilde{\alpha}=\widetilde{\beta}-\widetilde{\beta}=0.$$

\noindent We then obtain that $w\in H^0(L\otimes\mathfrak{I})$ is the unique element which sends $1+\widetilde{u}$ to $1+\widetilde{u}'$. The action is then simply transitive, and we have constructed $o_0$.\\

This then proves an analogue of the proposition given in \cite[$\mathsection$2.6]{goldman}. The other parts of the proof only use these three obstructions maps and do not directly use the structure of $L$. We can then follow exactly the same arguments in \cite[$\mathsection$2.11]{goldman} to obtain the result.

\end{proof}

\begin{defi}
    Two positively graded $\Gamma(\mathcal{P}re\mathcal{L}ie,-)$-algebras $L$ and $\overline{L}$ are \text{\normalfont quasi-isomorphic} if there exists a zig-zag of morphisms of $\Gamma(\mathcal{P}re\mathcal{L}ie,-)$-algebras
    $$L=L_{0}\longrightarrow L_{1}\longleftarrow \cdots\longrightarrow L_{m-1}\longleftarrow L_{m}=\overline{L}$$

    \noindent in which each morphism induces an isomorphism in cohomology.
\end{defi}

\begin{cor}
If $L$ and $\overline{L}$ are quasi-isomorphic, then for all local artinian $\textbf{K}$-algebras $A$, the groupoids $\text{\normalfont Deligne}(L,A)$ and $\text{\normalfont Deligne}(\overline{L},A)$ are equivalent. More precisely, we have a zig-zag of equivalence of groupoids
$$\text{\normalfont Deligne}(L,A)\longrightarrow \text{\normalfont Deligne}(L_{1},A)\longleftarrow \cdots\longrightarrow \text{\normalfont Deligne}(L_{m-1},A)\longleftarrow \text{\normalfont Deligne}(\overline{L},A)$$

\noindent which is natural in $A$.
\end{cor}

\section{Application in homotopy theory for operads}

The goal of this section is to establish Theorem \ref{theoremC}, which gives a computation of $\pi_{0}(\text{\normalfont Map}(B^{c}(\mathcal{C}),\mathcal{P}))$ where $\mathcal{C}$ is a $\Sigma_{*}$-cofibrant coaugmented cooperad, $\mathcal{P}$ an augmented operad and $B^{c}$ the cobar construction (see \cite{fresselivre} or \cite{loday} for a definition of this construction). In the case of a field of characteristic $0$, it can be expressed in terms of the Deligne groupoid with the structure of dg Lie algebra of ${\normalfont \text{Hom}}_\Sigma(\overline{\mathcal{C}},\overline{\mathcal{P}})$. We extend this result using a structure of $\Gamma(\mathcal{P}re\mathcal{L}ie,-)$-algebra that underlies this dg Lie algebra structure.\\

In $\mathsection$\ref{sec:31}, we define infinitesimal $k$-compositions and $k$-decompositions that generalize the usual infinitesimal composition and decomposition operations given in \cite[$\mathsection$6.1]{loday}. These operations will be used in the next section to write more easily weighted brace operations of the convolution operad.\\

In $\mathsection$\ref{sec:32}, we recall the definition of the convolution operad ${\normalfont \text{Hom}}(\mathcal{C},\mathcal{P})$, as given in \cite[$\mathsection$6.4.1]{loday}, and study the $\Gamma(\mathcal{P}re\mathcal{L}ie,-)$-algebra structure of $\text{\normalfont Hom}_\Sigma(\mathcal{C},\mathcal{P})$. In the same way that infinitesimal composition and decomposition can be used to express the pre-Lie algebra structure of the convolution operad (see \cite[Proposition 6.4.5]{loday}), we will use infinitesimal $k$-compositions and $k$-decompositions to compute weighted brace operations of the convolution operad.\\

In $\mathsection$\ref{sec:33}, we just use a cylinder object of $B^c(\mathcal{C})$ given by Fresse in \cite[$\mathsection$5.1]{fressecyl} to get our result: the quotient of $\text{\normalfont Hom}_\Sigma(\mathcal{C},\mathcal{P})$ by the gauge action gives $\pi_{0}(\text{\normalfont Map}(B^{c}(\mathcal{C}),\mathcal{P}))$.

\subsection{Infinitesimal compositions and decompositions of an operad and a cooperad}\label{sec:31}

We first introduce some definitions which will be useful for the computations.\\

Let $M$ and $N$ be two symmetric sequences such that $N(0)=0$. Recall that we have a monoidal structure on the category of symmetric sequences defined by
$$M\circ N(n)=\bigoplus_{k\geq 0} M(k)\otimes_{\Sigma_{k}}\left(\bigoplus_{i_{1}+\cdots+i_{k}=n} Ind_{\Sigma_{i_{1}}\times\cdots\times\Sigma_{i_{k}}}^{\Sigma_{n}}(N(i_{1})\otimes\cdots\otimes N(i_{k}))\right),$$

\noindent with as unit the symmetric sequence $I$ defined by 
$$I(n)=\left\{\begin{array}{lll}
     \mathbb{K} & \text{if }n=1  \\
    0 & \text{if }n\neq 1
\end{array}\right..$$

\noindent Every elements of $M\circ N(n)$ can be identified as a tree of the form:
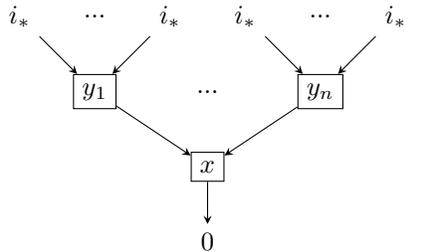
\begin{center}
\begin{tikzpicture}[baseline={([yshift=-.5ex]current bounding box.center)}]
    \node[draw] (i) at (0,0) {$x$};
    \node[draw] (1) at (-1.5,1) {$y_{1}$};
    \node (1b) at (-2.5,2) {$i_{1}^1$};
    \node (1bb) at (-0.5,2) {$i_{r_1}^1$};
    \node[draw] (2) at (1.5,1) {$y_{n}$};
    \node (2b) at (0.5,2) {$i_{1}^n$};
    \node (2bb) at (2.5,2) {$i_{r_n}^n$};
    \node (ab) at (-1.5,2) {$\cdots$};
    \node (abb) at (1.5,2) {$\cdots$};
    \node (a) at (0,1) {$\cdots$};
    \node (b) at (0,-1) {$0$};
    \draw[-stealth] (2b) -- (2);
    \draw[-stealth] (2bb) -- (2);
    \draw[-stealth] (1b) -- (1);
    \draw[-stealth] (1bb) -- (1);
    \draw[-stealth] (2) -- (i);
    \draw[stealth-] (i) -- (1);
    \draw [-stealth] (i) -- (b);
\end{tikzpicture}\ \ ,
\end{center}

\noindent where $x\in M(n)$, $y_{1}\in N(r_1),\ldots,y_{n}\in N(r_n)$, and where $1\leq i_p^q\leq\sum_i r_i$ are labels which represent a permutation of $\Sigma_{r_1+\cdots+r_n}$.\\

Note that we can write $M\circ N (n)$ without quotients by the group permutations by taking a choice of set of representatives. This set is given by pointed shuffle permutations (see conventions):
$$M\circ N(n)=\bigoplus_{k\geq 0}M(k)\otimes\left(\bigoplus_{i_1+\cdots+i_k=n}N(i_1)\otimes \cdots\otimes N(i_k)\otimes\mathbb{K}[Sh_\ast(i_1,\ldots,i_k)]\right).$$

We now generalize the definition of the infinitesimal composition/decomposition defined in \cite{loday}, in order to write some formulas in a more convenient way.

\begin{defi}\label{kcomp}
    Let $M$ and $N$ be two symmetric sequences. For all $k\geq 0$, we define a new symmetric sequence denoted by $M\circ_{(k)}N$ called the \text{\normalfont $k$-infinitesimal composite} of $M$ and $N$ defined, in each arity $n$, as the submodule of $M\circ (I\oplus N)(n)$ spanned by trees where exactly $k$ elements at level $2$ are in ${N}$, and the others in $I$.
\end{defi}

Let $M,M',N$ and $N'$ be symmetric sequences. One can easily check that if we have morphisms of symmetric sequences $f:M\longrightarrow M'$ and ${g}:{N}\longrightarrow{N}'$, then we have a morphism $f\circ_{(k)}g:M\circ_{(k)}N\longrightarrow M'\circ_{(k)}N'$ induced by $f\circ (id_I\oplus g)$.\\

Let $\mathcal{P}$ be an operad with composition $\gamma:\mathcal{P}\circ\mathcal{P}\longrightarrow\mathcal{P}$ and unit $\eta:I\longrightarrow\mathcal{P}$, and let $\mathcal{C}$ a cooperad with coproduct $\Delta:\mathcal{C}\longrightarrow\mathcal{C}\circ\mathcal{C}$ and counit $\varepsilon:\mathcal{C}\longrightarrow I$. We will suppose that $\mathcal{P}$ is augmented, i.e. the unit $\eta:I\longrightarrow\mathcal{P}$ admits a retraction $\pi:\mathcal{P}\longrightarrow I$. We then have that there exists a symmetric sequence $\overline{\mathcal{P}}$ with $\mathcal{P}\simeq I\oplus\overline{\mathcal{P}}$ such that the first projection on $\mathcal{P}$ is given by $\pi$. Similarly, we suppose that $\mathcal{C}$ is coaugmented, i.e. the counit $\varepsilon:\mathcal{C}\longrightarrow I$ admits a section $s:I\longrightarrow\mathcal{C}$. We then have that there exists a symmetric sequence $\overline{\mathcal{C}}$ with $\mathcal{C}\simeq I\oplus\overline{\mathcal{C}}$ such that the first projection is given by $\varepsilon$.\\

In the following, we assume that $\mathcal{C}(0)=\mathcal{P}(0)=0$ and $\mathcal{C}(1)=\mathcal{P}(1)=\mathbb{K}$.\\

We give an extension of the usual infinitesimal composition and decomposition operations given in \cite[$\mathsection$6.1]{loday} for $k=1$.

\begin{defi}\label{compo}
    Let $k\geq 1$.
    \begin{itemize}
        \item[-] We define {\normalfont the infinitesimal $k$-composition} in $\mathcal{P}$ as
        \begin{center}
\begin{tikzcd}
\gamma_{(k)}:\overline{\mathcal{P}}\circ_{(k)}\overline{\mathcal{P}}(n)\arrow[r] & \mathcal{P}\circ\mathcal{P}(n)\arrow[r, "\gamma"] & {\mathcal{P}}(n),
\end{tikzcd}
\end{center}

\noindent where the first map is the inclusion of $\overline{\mathcal{P}}\circ_{(k)}\overline{\mathcal{P}}$ in $\mathcal{P}\circ\mathcal{P}$.

    \item[-] We define the {\normalfont infinitesimal $k$-decomposition} in $\mathcal{C}$ as
    \begin{center}
\begin{tikzcd}
\Delta_{(k)}:{\mathcal{C}}(n)\arrow[r, "\Delta"] & \mathcal{C}\circ\mathcal{C}(n)\arrow[r] & \overline{\mathcal{C}}\circ_{(k)}\overline{\mathcal{C}}(n),
\end{tikzcd}
\end{center}

\noindent where the last map is the projection of $\mathcal{C}\circ\mathcal{C}$ onto $\overline{\mathcal{C}}\circ_{(k)}\overline{\mathcal{C}}$.
    \end{itemize}
\end{defi}

Because $\mathcal{C}$ is coaugmented, we have that the coproduct $\Delta:\mathcal{C}\longrightarrow\mathcal{C}\circ\mathcal{C}$ preserves the isomorphism $\mathcal{C}\simeq I\oplus\overline{\mathcal{C}}$ in the following sense. We have the isomorphism
$$\mathcal{C}\circ\mathcal{C}\simeq I\circ I\oplus \overline{\mathcal{C}}\circ I\oplus I\circ\overline{\mathcal{C}}\oplus\bigoplus_{k\geq 1}\overline{\mathcal{C}}\circ_{(k)}\overline{\mathcal{C}}.$$

\noindent Then, we get that the restriction of $\Delta$ on $I$ and on $\overline{\mathcal{C}}$ are such that $\Delta:I\longrightarrow I\circ I$ and $\Delta:\overline{\mathcal{C}}\longrightarrow \overline{\mathcal{C}}\circ I\oplus I\circ\overline{\mathcal{C}}\oplus\bigoplus_{k\geq 1}\overline{\mathcal{C}}\circ_{(k)}\overline{\mathcal{C}}$. We can then define the infinitesimal $k$-decompositions on $\overline{\mathcal{C}}$ by
\begin{center}
\begin{tikzcd}
\Delta_{(0)}:\overline{\mathcal{C}}\arrow[r, "\Delta"] & \displaystyle\overline{\mathcal{C}}\circ I\oplus I\circ\overline{\mathcal{C}}\oplus\bigoplus_{k\geq 1}\overline{\mathcal{C}}\circ_{(k)}\overline{\mathcal{C}}\arrow[r, two heads] & \overline{\mathcal{C}}\circ I\oplus I\circ\overline{\mathcal{C}},\\
\Delta_{(k)}:\overline{\mathcal{C}}\arrow[r, "\Delta"] & \displaystyle\overline{\mathcal{C}}\circ I\oplus I\circ\overline{\mathcal{C}}\oplus\bigoplus_{k\geq 1}\overline{\mathcal{C}}\circ_{(k)}\overline{\mathcal{C}}\arrow[r, two heads] & \displaystyle\overline{\mathcal{C}}\circ_{(k)}\overline{\mathcal{C}},
\end{tikzcd}
\end{center}

\noindent for all $k\geq 1$.\\


\subsection{$\Gamma(\mathcal{P}re\mathcal{L}ie,-)$-algebra structure of the convolution operad}\label{sec:32}

Let $M$ and $N$ be two symmetric sequences of differential graded $\mathbb{K}$-modules. We define a new symmetric sequence $\text{\normalfont Hom}(M,N)$ in dg $\mathbb{K}$-modules by $\text{\normalfont Hom}(M,N)(n)=\text{\normalfont Hom}(M(n),N(n))$, the differential graded module formed by the homogeneous morphisms $f:M(n)\longrightarrow N(n)$. The differential on $\text{\normalfont Hom}(M,N)$ is given by
$$d(f)=d_{M}\circ f-(-1)^{deg(f)}f\circ d_{N},$$

\noindent for all $f\in\text{\normalfont Hom}(M,N)$. The action of $\Sigma_{n}$ on $\text{\normalfont Hom}(M(n),N(n))$ is defined by
$$\forall x\in M(n), (\sigma\cdot f)(x)=\sigma\cdot f(\sigma^{-1}\cdot x),$$

\noindent for all $\sigma\in\Sigma_{n}$.

\begin{prop}\text{\normalfont (see \cite{loday})} Let $\mathcal{C}$ be a cooperad and $\mathcal{P}$ be an operad. Then $\text{\normalfont Hom}(\mathcal{C},\mathcal{P})$ has the structure of a dg operad called the {\normalfont convolution operad} of $\mathcal{C}$ and $\mathcal{P}$.
\end{prop}

We recall the operad structure on $\text{\normalfont Hom}(\mathcal{C},\mathcal{P})$. For $f\in \text{\normalfont Hom}(\mathcal{C},\mathcal{P})(k)$, $g_{1}\in \text{Hom}(\mathcal{C},\mathcal{P})(i_{1}),\ldots, g_{k}\in \text{Hom}(\mathcal{C},\mathcal{P})(i_{k})$ the composition $\gamma(f\otimes g_{1}\otimes \cdots\otimes g_{k})$ is given by the composite

\[\begin{tikzcd}
	{\mathcal{C}(n)} & {\mathcal{C}\circ\mathcal{C}(n)} & {\mathcal{C}(k)\otimes\mathcal{C}(i_{1})\otimes \cdots\otimes\mathcal{C}(i_{k})\otimes\mathbb{K}[id]} \\
	&& {\mathcal{P}(k)\otimes \mathcal{P}(i_{1})\otimes \cdots\otimes \mathcal{P}(i_{k})\otimes\mathbb{K}[id]} & {\mathcal{P}\circ\mathcal{P}(n)} & {\mathcal{P}(n)}
	\arrow["\Delta", from=1-1, to=1-2]
	\arrow[two heads, from=1-2, to=1-3]
	\arrow["{f\otimes g_{1}\otimes\cdots\otimes g_{k}\otimes id}", from=1-3, to=2-3]
	\arrow[hook, from=2-3, to=2-4]
	\arrow["\gamma", from=2-4, to=2-5]
\end{tikzcd}\]

\noindent where $\displaystyle n=\sum_{p}i_{p}$.\\

We have $\text{\normalfont Hom}_{\Sigma_{n}}(\mathcal{C}(n),\mathcal{P}(n))=\text{\normalfont Hom}(\mathcal{C}(n),\mathcal{P}(n))^{\Sigma_n}$, and we set
$$\text{\normalfont Hom}_{\Sigma}(\mathcal{C},\mathcal{P})=\prod_{n\geq 1}\text{\normalfont Hom}_{\Sigma_{n}}(\mathcal{C}(n),\mathcal{P}(n));$$

$$\text{\normalfont Hom}_\Sigma(\overline{\mathcal{C}},\overline{\mathcal{P}})=\prod_{n\geq 2}\text{\normalfont Hom}_{\Sigma_n}({\mathcal{C}}(n),\mathcal{P}(n))\subset\text{\normalfont Hom}_\Sigma(\mathcal{C},\mathcal{P}).$$


\noindent Then, according to Corollary \ref{prod}, we have that $\text{\normalfont Hom}_{\Sigma}(\overline{\mathcal{C}},\overline{\mathcal{P}})$ is endowed with a complete $\Gamma(\mathcal{P}re\mathcal{L}ie,-)$-algebra structure. We also have the isomorphism


$$\text{\normalfont Hom}_\Sigma(\mathcal{C},\mathcal{P})\simeq\mathbb{K}\oplus\text{\normalfont Hom}_{\Sigma}(\overline{\mathcal{C}},\overline{\mathcal{P}}),$$

\noindent so that any morphism in $\text{\normalfont Hom}_\Sigma(\overline{\mathcal{C}},\overline{\mathcal{P}})$ can be identified with a morphism in $\text{\normalfont Hom}_\Sigma(\mathcal{C},\mathcal{P})$ which is $0$ on $\mathcal{C}(1)=\mathbb{K}$.\\

We can explicitly describe the weighted braces with one input of $\text{\normalfont Hom}_\Sigma(\overline{\mathcal{C}},\overline{\mathcal{P}})$ in terms of infinitesimal decompositions and compositions.

\begin{lm}
Let $\overline{f},\overline{g}\in \text{\normalfont Hom}_{\Sigma}(\overline{\mathcal{C}},\overline{\mathcal{P}})$. Then $\overline{f}\{\overline{g}\}_{k}$ is given by the composite
\begin{center}
\begin{tikzcd}
\overline{\mathcal{C}}\arrow[r, "\Delta_{(k)}"] & \overline{\mathcal{C}}\circ_{(k)}\overline{\mathcal{C}}\arrow[r, "\overline{f}\circ_{(k)}\overline{g}"] & \overline{\mathcal{P}}\circ_{(k)}\overline{\mathcal{P}}\arrow[r, "\gamma_{(k)}"] & \overline{\mathcal{P}}
\end{tikzcd}.
\end{center}
\end{lm}

In particular, we retrieve the well-known pre-Lie algebra structure on $\text{\normalfont Hom}_\Sigma(\overline{\mathcal{C}},\overline{\mathcal{P}})$ given by the composite
\begin{center}
\begin{tikzcd}
\overline{\mathcal{C}}\arrow[r, "\Delta_{(1)}"] & \overline{\mathcal{C}}\circ_{(1)}\overline{\mathcal{C}}\arrow[r, "\overline{f}\circ_{(1)}\overline{g}"] & \overline{\mathcal{P}}\circ_{(1)}\overline{\mathcal{P}}\arrow[r, "\gamma_{(1)}"] & \overline{\mathcal{P}}
\end{tikzcd}
\end{center}

\noindent as shown in \cite[Proposition 6.4.5]{loday} (note that here we consider left actions).\\

\begin{proof}
    By definition of the $\Gamma(\mathcal{P}re\mathcal{L}ie,-)-$algebra structure of $\text{\normalfont Hom}_\Sigma(\overline{\mathcal{C}},\overline{\mathcal{P}})$ (see Corollary \ref{prod}), it is sufficient to prove it on $\bigoplus_{r\geq 2}\text{\normalfont Hom}_{\Sigma_r}({\mathcal{C}}(r),\mathcal{P}(r))$. We write $\overline{g}=\overline{g_1}+\cdots+\overline{g_p}$ where $\overline{g_i}\in\text{\normalfont Hom}_{\Sigma_{n_i}}({\mathcal{C}}(n_i),{\mathcal{P}}(n_i))$ with $n_i\neq n_j$ whenever $i\neq j$. Since the identity we need to prove is linear in $\overline{f}$, we can suppose that $\overline{f}\in\text{\normalfont Hom}_{\Sigma_n}({\mathcal{C}}(n),{\mathcal{P}}(n))$. We then have that
    $$\overline{f}\{\overline{g}\}_k=\sum_{r_1+\cdots+r_p=k}\overline{f}\{\overline{g_1},\ldots,\overline{g_p}\}_{r_1,\ldots,r_p}.$$

    \noindent If we denote by $\gamma$ the operadic composition in the operad $\text{\normalfont Hom}(\mathcal{C},\mathcal{P})$, and if we set $\overline{h_1},\ldots,\overline{h_k}=\underbrace{\overline{g_1},\ldots,\overline{g_1}}_{r_1},\ldots,\underbrace{\overline{g_p},\ldots,\overline{g_p}}_{r_p}$ and ${s_1},\ldots,{s_k}=\underbrace{n_1,\ldots,n_1}_{r_1},\ldots,\underbrace{n_p,\ldots,n_p}_{r_p}$, then by definition of the $\Gamma(\mathcal{P}re\mathcal{L}ie,-)$-algebra structure of $\bigoplus_{r\geq 2}\text{\normalfont Hom}_{\Sigma_r}(\mathcal{C}(r),\mathcal{P}(r))$ (see Proposition \ref{symopealg}), we have
    $$\overline{f}\{\overline{g_1},\ldots,\overline{g_p}\}_{r_1,\ldots,r_p}=\sum_{\sigma\in Sh(r_1,...,r_p)}\sum_{\substack{1\leq i_1<\cdots<i_k\leq n\\\omega\in Sh_\ast(1,\ldots,\underset{i_1}{{s}_{\sigma^{-1}(1)}},\ldots,\underset{i_k}{{s}_{\sigma^{-1}(k)}},\ldots,1)}}\pm\omega\cdot\gamma(\overline{f}(1,\ldots,\underset{i_1}{\overline{g}_{\sigma^{-1}(1)}},\ldots,\underset{i_k}{\overline{g}_{\sigma^{-1}(k)}},\ldots,1)).$$

    \noindent For a given $\sigma\in Sh(r_1,\ldots,r_p)$ and $\omega\in Sh_\ast(1,\ldots,\underset{i_1}{{s}_{\sigma^{-1}(1)}},\ldots,\underset{i_k}{{s}_{\sigma^{-1}(k)}},\ldots,1)$, we can see the corresponding term in the sum as the composite

    \[\begin{tikzcd}
	{{\mathcal{C}}(m)} & {\mathcal{C}\circ\mathcal{C}(m)} & {\mathcal{C}(n)\otimes\mathcal{C}(s_{\sigma^{-1}(1)})\otimes\cdots\otimes\mathcal{C}(s_{\sigma^{-1}(k)})\otimes\mathbb{K}[\omega]} \\
	&& {\mathcal{P}(n)\otimes \mathcal{P}(s_{\sigma^{-1}(1)})\otimes\cdots\otimes \mathcal{P}(s_{\sigma^{-1}(k)})\otimes\mathbb{K}[\omega]} & {\mathcal{P}\circ\mathcal{P}(m)} & {{\mathcal{P}}(m)}
	\arrow["\Delta", from=1-1, to=1-2]
	\arrow[two heads, from=1-2, to=1-3]
	\arrow["{\overline{f}\otimes \overline{h}_{\sigma^{-1}(1)}\otimes\cdots\otimes \overline{h}_{\sigma^{-1}(k)}\otimes id}", from=1-3, to=2-3]
	\arrow[hook, from=2-3, to=2-4]
	\arrow["\gamma", from=2-4, to=2-5]
\end{tikzcd}\]

\noindent where $m=n+(n_1-1)r_1+\cdots+(n_p-1)r_p$, and where we omits some unit elements. Summing over the $\omega$'s and using the decomposition of $\mathcal{C}\circ\mathcal{C}(m)$ and $\mathcal{P}\circ\mathcal{P}(m)$ just before Definition \ref{kcomp} will then give the composite
\[\begin{tikzcd}
	{{\mathcal{C}}(m)} & {\mathcal{C}\circ\mathcal{C}(m)} &&& {\mathcal{P}\circ\mathcal{P}(m)} & {{\mathcal{P}}(m)}
	\arrow["\Delta", from=1-1, to=1-2]
	\arrow["{\overline{f}\otimes\overline{h}_{\sigma^{-1}(1)}\otimes\cdots\otimes\overline{h}_{\sigma^{-1}(k)}}", from=1-2, to=1-5]
	\arrow["\gamma", from=1-5, to=1-6]
\end{tikzcd}\]

\noindent where we omit, again, the unit elements. Note that, since $\overline{f}$ and the $\overline{h}_i$'s are $0$ on $\mathcal{C}(1)$, this composite is also $0$ on $\mathcal{C}(1)$. Summing on all $\sigma\in Sh(r_1,\ldots,r_p)$ will give the desired composite.
\end{proof}

\begin{thm}
The circular product of two elements $f=1+\overline{f}, g=1+\overline{g}$
 of the gauge group of $\text{\normalfont Hom}_{\Sigma}(\overline{\mathcal{C}},\overline{\mathcal{P}})$ is given by
\begin{center}
\begin{tikzcd}
f\circledcirc g:\mathcal{C}\arrow[r, "\Delta"] & \mathcal{C}\circ\mathcal{C}\arrow[r, "f\circ g"] & \mathcal{P}\circ\mathcal{P}\arrow[r, "\gamma"] & \mathcal{P}
\end{tikzcd}.
\end{center}
\end{thm}

\begin{proof}
Because $f_{|I}=g_{|I}=1$, we have that $(f\circledcirc g)_{|I}=1$. We thus need to show the equality on $\overline{\mathcal{C}}$. Recall that we have infinitesimal decompositions on $\overline{\mathcal{C}}$ denoted by $\Delta_{(0)}$ and $\Delta_{(k)}$ for $k\geq 1$ such that $\Delta_{|\overline{\mathcal{C}}}=\Delta_{(0)}\oplus\bigoplus_{k\geq 1}\Delta_{(k)}$. The map $\Delta_{(0)}$ will give $\overline{f}+\overline{g}$, and each $\Delta_{(k)}$ will give $\overline{f}\{\overline{g}\}_{k}$ according to the previous lemma. We thus have that the composite in the statement of the theorem gives
$$1+\overline{f}+\sum_{k\geq 0}\overline{f}\{\overline{g}\}_{k}$$

\noindent which is exactly $f\circledcirc g$.
\end{proof}

\subsection{Computation of $\pi_0(\text{\normalfont Map}(B^c(\mathcal{C}),\mathcal{P}))$}\label{sec:33}

We now extend the computation of $\pi_0(\text{\normalfont Map}(B^c(\mathcal{C}),\mathcal{P}))$ on a field $\mathbb{K}$ with positive characteristic. In this last section, we chose to work with a homological convention to follow the conventions in the literature. Note that this does not change anything on the results of the previous sections.\\

Recall that we can give an explicit cylinder object for $B^c(\mathcal{C})$, where $B^c$ is the cobar construction of $\mathcal{C}$, when $\mathcal{C}$ is $\Sigma_{*}$-cofibrant (see for instance \cite{fressecyl} or \cite{loday}). Explicitly, let $K=\mathbb{K}\sigma^{0}\oplus\mathbb{K}\sigma^{1}\oplus\mathbb{K}\sigma^{01}$ where $|\sigma^{0}|=|\sigma^{1}|=-1$, $|\sigma^{01}|=0$ and $d(\sigma^{01})=\sigma^{1}-\sigma^{0}$. Then there exists a derivation of operads $\partial$ such that the free dg operad $(\mathcal{F}(K\otimes\overline{\mathcal{C}}),\partial)$ is a cylinder object for $B^c(\mathcal{C})$. We refer to \cite[$\mathsection$5.1]{fressecyl} for an explicit construction of $\partial$ and a proof of the previous statement.

\begin{thm}
Suppose that $\mathcal{C}$ is $\Sigma_{*}$-cofibrant. We then have a bijection:
$$\pi_{0}(\text{\normalfont Map}(B^c(\mathcal{C}),\mathcal{P}))\simeq \pi_0\text{\normalfont Deligne}(\text{\normalfont Hom}_{\Sigma}(\overline{\mathcal{C}},\overline{\mathcal{P}})).$$
\end{thm}

\begin{proof} Recall from Fresse (see \cite[Theorem 3.2.14]{fresselivre} for instance) that $\pi_{0}(\text{\normalfont Map}(B^c(\mathcal{C}),\mathcal{P}))\simeq (Mor(B^c(\mathcal{C}),\mathcal{P}),\sim_{h})$ where $\sim_{h}$ is the homotopy relation in the category of symmetric operads. Recall also that the data of a Maurer-Cartan element $\alpha$ in $\text{\normalfont Hom}_{\Sigma}(\overline{\mathcal{C}},\overline{\mathcal{P}})$ is equivalent to give a morphism of operads $\phi_{\alpha}$ from $B^c(\mathcal{C})$ to $\mathcal{P}$ (see \cite{fressecyl} or \cite{loday}). We just need to show that the action of the gauge group on the Maurer-Cartan set of $\text{\normalfont Hom}_{\Sigma}(\overline{\mathcal{C}},\overline{\mathcal{P}})$ from one Maurer-Cartan $\alpha$ to an other one $\beta$ is equivalent to give a homotopy from $\phi_{\alpha}$ to $\phi_{\beta}$.\\

Let $1+\lambda$ be an element of the gauge group. We define a morphism $h:\text{\normalfont Cyl}(B^c(\mathcal{C}))\longrightarrow\mathcal{P}$ via $h:K\otimes\overline{\mathcal{C}}\longrightarrow\mathcal{F}(K\otimes\overline{\mathcal{C}})$ by setting
$$h(\sigma^{0}\otimes\gamma)=\alpha(\gamma),$$
$$h(\sigma^{1}\otimes\gamma)=\beta(\gamma),$$
$$h(\sigma^{01}\otimes\gamma)=\lambda(\gamma),$$

\noindent where $\gamma$ is some element of $\overline{\mathcal{C}}$. We claim that $(1+\lambda)\cdot\alpha=\beta$ if and only if $h$ is a homotopy from $\phi_{\alpha}$ to $\phi_{\beta}$. Accordingly, we must prove the equivalence
$$d(\lambda)=\alpha+\lambda\{\alpha\}-\beta\circledcirc (1+\lambda)\Leftrightarrow d(h)=0$$

\noindent where $d$ is the differential of $Mor(B^{c}(\mathcal{C}),\mathcal{P})$.\\

Because $\alpha$ and $\beta$ are Maurer-Cartan elements, and by definition of $\partial$, the second equality is always satisfied for $\sigma^{\varepsilon}\otimes\gamma$ with $\varepsilon=0,1$ and $\gamma\in\overline{\mathcal{C}}$. We just need to check this equality on terms $\sigma^{01}\otimes\gamma$ for any $\gamma\in\overline{\mathcal{C}}$:
\begin{center}
$\begin{array}{lll}
d(h)(\sigma^{01}\otimes\gamma) & = & d(h(\sigma^{01}\otimes\gamma))-h(\partial(\sigma^{01}\otimes \gamma))\\
& = & d(\lambda(\gamma))-\lambda(d(\gamma))-\alpha(\gamma)+\beta(\gamma)-\lambda\{\alpha\}_{1}(\gamma)+(\beta\circledcirc(1+\lambda)(\gamma)-\beta(\gamma))\\
& = & d(\lambda)(\gamma)-\alpha(\gamma)-\lambda\{\alpha\}_{1}(\gamma)+\beta\circledcirc(1+\lambda)(\gamma).\\
\end{array}$
\end{center}

We then have the desired equivalence.
\end{proof}

\printbibliography[heading=bibintoc,title={References}]

\end{document}